\documentclass[a4paper,11pt,fleqn]{article}
\usepackage{amsmath,amssymb,amsthm,graphicx,subfigure,float,caption,epstopdf,tabularx,color, bm,amsfonts}
\usepackage{tikz}
\usepackage{booktabs}
\usetikzlibrary{calc,intersections,math}
\usepackage{caption}
\usepackage{subcaption}
\usepackage{float}
\usepackage{amsmath}
\usetikzlibrary{math}
\usepackage[top=1in, bottom=1in, left=1.25in, right=1.25in]{geometry}
\numberwithin{equation}{section}
\newtheorem{theorem}{Theorem}[section]
\newtheorem{lemma}[theorem]{Lemma}

\newtheorem{remark}[theorem]{Remark}

\newcommand{\normmm}[1]{{\left\vert\kern-0.25ex\left\vert\kern-0.25ex\left\vert #1
		\right\vert\kern-0.25ex\right\vert\kern-0.25ex\right\vert}}
	
	\newdimen\cellthresh
	\newdimen\celldist
\begin{document}
\title{A Nonconforming Finite Element Method for Elliptic Interface Problems on Locally Anisotropic Mixed Meshes}
\author{
	Chenchen Geng$^{1}$ \quad
	Hua Wang$^{1}$\footnote{Correspondence author. E-mail addresses: wanghua@xtu.edu.cn (H. Wang).\\
		\textbf{Funding:} This research is supported by NSFC project 12101526 and Young Elite Scientists Sponsorship Program by CAST 2023QNRC001.}
	\quad
	Qichen Zhang$^2$\\
	{\small $^1$ School of Mathematics and Computational Science, Xiangtan University, Xiangtan 411105, China}\\
     \small	$^2$ School of Mathematical Sciences, Peking University, Beijing 100871, China \\
}

	\date{}
	\maketitle
	
\begin{abstract}
	We propose a new nonconforming \(P_1\) finite element method for elliptic interface problems. The method is constructed on a locally anisotropic mixed mesh, which is generated by fitting the interface through a simple connection of intersection points on an interface-unfitted background mesh, as introduced in \cite{Hu2021optimal}. We first establish interpolation error estimates on quadrilateral elements satisfying the regular decomposition property (RDP). Another key contribution is a novel consistency error analysis for nonconforming elements, which removes the quasi-regularity assumption commonly required in existing approaches. Numerical results confirm the theoretical convergence rates and demonstrate the robustness and accuracy of the proposed method.	
\end{abstract}

	
\section{Introduction}
In this paper, we present a finite element method for solving the elliptic interface problem:
\begin{equation}\label{interface-problem}
	\begin{aligned}
		-\nabla \cdot (\beta \nabla u) &= f \quad \text{in} \ \Omega_{1} \cup \Omega_{2}, \\
		u &= 0 \quad \text{on} \ \partial \Omega, \\
		[\![u]\!] &= 0 \quad \text{on} \ \Gamma, \\
		\left[\!\!\left[ \beta \frac{\partial u}{\partial \bm{n}} \right]\!\!\right] &= 0 \quad \text{on} \ \Gamma,
	\end{aligned}
\end{equation}
where the discontinuous diffusion coefficient $\beta$ is defined as
\begin{equation}\label{beta-def}
	\beta = 
	\begin{cases} 
		\beta_1 & \text{in} \ \Omega_1, \\
		\beta_2 & \text{in} \ \Omega_2,
	\end{cases}
\end{equation}
with $\beta_1, \beta_2 > 0$ being positive constants.

Interface problems arise in diverse scientific and engineering fields, including composite material analysis \cite{Kafafy2005three}, fluid-structure interaction \cite{Gerstenberger2008an,Tezduyar2006space}, and multiphase flow simulations \cite{Chen1997fully}. These applications feature solutions with reduced regularity across material interfaces, posing significant challenges for numerical methods. The Finite Element Method (FEM) has emerged as a primary computational tool, categorized into fitted and unfitted approaches based on mesh-interface alignment.

Unfitted mesh methods are particularly effective for interface problems involving complex geometries due to their flexibility in managing irregular interface structures. These methods have undergone significant advancements aimed at improving accuracy, computational efficiency, and adaptability to complex interfaces. Prominent examples include the Immersed Finite Element Method (IFEM) and the Extended Finite Element Method (XFEM), extensively studied and validated in numerous works (see, e.g., \cite{Li1998, Li2003, An2014, Chou2012immersed, Lin2015, Ji2023,Wang2019a,Zhang2019strongly}). Additionally, other promising approaches like the Generalized Finite Element Method (GFEM) have also been proposed, see \cite{Fries2010extended} for an overview. IFEM typically modifies finite element basis functions to explicitly satisfy interface conditions, whereas XFEM introduces penalization terms into the variational formulation to weakly enforce these conditions, known as interior penalty or Nitsche's methods (see \cite{Hansbo2002unfitted, Li2003, burman2015cutfem, Hansbo2014cut, Cao2022extended, Fries2010extended}). For example, Chen et al. \cite{Chen2023an} combined XFEM with a novel mesh generation strategy, effectively merging small interface elements with neighboring elements.

Another widely investigated strategy involves refining the unfitted mesh near interfaces to construct locally fitted or anisotropic meshes. Previous studies have demonstrated significant advancements using this approach (see \cite{Chen2009the, Xu2016optimal, chen2017interface, Hu2021optimal}). Chen et al. \cite{Chen2009the} generated intermediate fitted meshes by subdividing interface tetrahedra into smaller tetrahedra through the latest vertex bisection algorithm, preserving mesh quality throughout adaptive refinement. Xu et al. \cite{Xu2016optimal} proposed linear finite element schemes for diffusion and Stokes equations on interface-fitted grids satisfying the maximal angle condition. Similarly, Chen et al. \cite{chen2017interface} developed methods for semi-structured, interface-fitted mesh generation in two and three dimensions, leveraging virtual element methods to solve elliptic interface problems.

However, refined elements adjacent to interfaces frequently violate minimal angle conditions (shape regularity), complicating error analysis and numerical stability. Despite these challenges, this refinement approach remains prevalent due to its adaptability in handling complex interface geometries.

Most unfitted methods encounter significant difficulties when employing nonconforming elements. In Nitsche-type XFEM approaches, the weak continuity across cut edges is compromised, requiring penalty terms to stabilize consistency errors (see \cite{Wang2019a}). For IFEM, despite preserving weak continuity, the inherent solution singularities at interfaces cause consistency errors to degrade by half-order accuracy compared to interpolation errors (see \cite{Ji2023analysis}).

This work introduces a novel nonconforming finite element method on the local anisotropic hybrid meshes developed by Hu and Wang \cite{Hu2021optimal}. Our approach treats interface elements as macro-elements and constructs piecewise linear nonconforming finite elements that exactly recover the Park-Shen element \cite{Park2003P1} when restricted to quadrilateral element. To address the limitations of standard integral mean operators for interpolation estimates on anisotropic elements—where accuracy deteriorates due to shape dependency as shown in \cite{Mao2005convergence}—we introduce vertex-based averaging operators. We rigorously prove that these novel operators deliver optimal-order interpolation error estimates, independent of element anisotropy.

Furthermore, for consistency error estimation, we eliminate conventional shape regularity requirements. Our analysis accommodates meshes containing polygonal elements, provided each element can be subdivided into triangular sub-elements satisfying the maximum angle condition. To our knowledge, prior consistency error analyses were limited to tensor-product grids (e.g., right triangles and rectangles) \cite{Apel2001cr,Mao2005convergence}, with no existing results for such general anisotropic meshes.

The core of our approach lies in a straightforward yet effective directional decomposition: we split the consistency error within each element into separate x-directional and y-directional components, develop tailored estimation techniques for each directional term, then systematically combine these directional estimates. This decomposition strategy enables precise control of optimal-order consistency errors, naturally accommodating both solution singularities at interfaces and local mesh anisotropy.

The remainder of this paper is organized as follows. Section 2 introduces essential notation and preliminary concepts. In Section 3, the core results are presented. Specifically, Section 3.1 constructs the finite element spaces defined on interface elements. Subsequently, Section 3.2 formulates the weak formulations, followed by Section 3.3, which conducts the interpolation error analysis. Section 3.4 addresses the consistency error, and Section 3.5 presents the finite element error analysis. Numerical experiments validating our theoretical results are detailed in Section 4. Finally, conclusions and future research directions are discussed in Section 5.
\section{Notation}
	For integer $r\geq 0$, define the piecewise $H^{r}$ Sobolev space
	\begin{equation*}
		H^{r}(\Omega_{1}\cup\Omega_{2})=\{v\in L^2(\Omega);v|_{\Omega_{i}}\in H^{r}(\Omega_{i}),i=1,2\},
	\end{equation*}
	equipped with the norm and semi-norm
		\begin{eqnarray*}\label{norm}
			\begin{aligned}
				\|v\|_{H^{r}(\Omega_{1}\cup\Omega_{2})}&=(\|v\|_{H^{r}(\Omega_1)}^{2}
				+\|v\|_{H^{r}(\Omega_2)}^{2})^{1/2},\\
				|v|_{H^{r}(\Omega_{1}\cup\Omega_{2})}&=(|v|_{H^{r}(\Omega_1)}^{2}
				+|v|_{H^{r}(\Omega_2)}^{2})^{1/2}.
			\end{aligned}
		\end{eqnarray*}
	Furthermore, let $\tilde{H}^{r}(\Omega_1\cup\Omega_2)=H^{1}_{0}(\Omega)\cap H^{r}(\Omega_{1}\cup\Omega_{2}).$
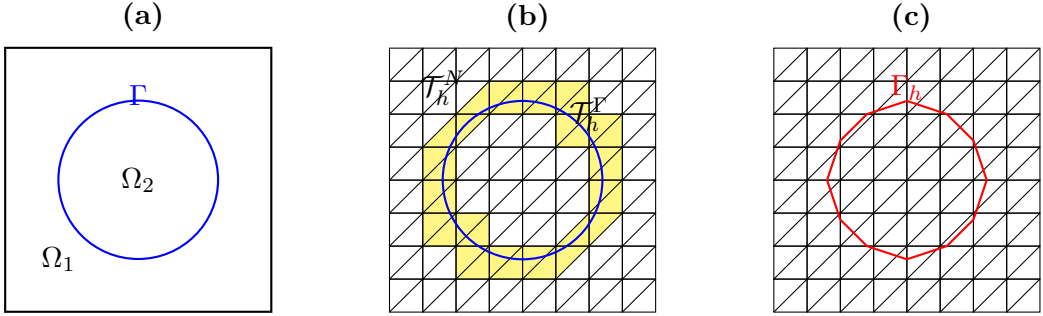
\begin{figure}[ht]
	\centering
	\begin{minipage}[b]{0.31\textwidth}
		\centering
		\textbf{(a)}\\[1ex]
		\begin{tikzpicture}[scale=3.5]
			\draw[black, thick] (0,0) rectangle (1,1);
			\draw[blue, thick] (0.5,0.5) circle[radius=0.3];
			\node at (0.2,0.2) {$\Omega_{1}$};
			\node at (0.5,0.5) {$\Omega_{2}$};
			\node[blue] at (0.5,0.82) {$\Gamma$};
		\end{tikzpicture}
		\label{subfig:domain}
	\end{minipage}
	\hfill
	\begin{minipage}[b]{0.31\textwidth}
		\centering
		\textbf{(b)}\\[1ex]
		\begin{tikzpicture}[scale=3.5]
			\def\n{8}
			\pgfmathsetmacro{\h}{1/\n}
			\foreach \i in {0,...,7} {
				\foreach \j in {0,...,7} {
					\pgfmathsetmacro{\x}{\i/\n}
					\pgfmathsetmacro{\y}{\j/\n}
					\coordinate (A) at (\x,\y);
					\coordinate (B) at (\x+\h,\y);
					\coordinate (C) at (\x+\h,\y+\h);
					\coordinate (D) at (\x,\y+\h);
					\pgfmathsetmacro{\phiA}{(\x-0.5)^2+(\y-0.5)^2-0.09}
					\pgfmathsetmacro{\phiB}{(\x+\h-0.5)^2+(\y-0.5)^2-0.09}
					\pgfmathsetmacro{\phiC}{(\x+\h-0.5)^2+(\y+\h-0.5)^2-0.09}
					\pgfmathsetmacro{\phiD}{(\x-0.5)^2+(\y+\h-0.5)^2-0.09}
					\pgfmathsetmacro{\minABC}{min(min(\phiA,\phiB),\phiC)}
					\pgfmathsetmacro{\maxABC}{max(max(\phiA,\phiB),\phiC)}
					\ifdim\minABC pt<0pt
					\ifdim\maxABC pt>0pt
					\fill[yellow!60] (A) -- (B) -- (C) -- cycle;
					\fi
					\fi
					\pgfmathsetmacro{\minACD}{min(min(\phiA,\phiC),\phiD)}
					\pgfmathsetmacro{\maxACD}{max(max(\phiA,\phiC),\phiD)}
					\ifdim\minACD pt<0pt
					\ifdim\maxACD pt>0pt
					\fill[yellow!60] (A) -- (C) -- (D) -- cycle;
					\fi
					\fi
				}
			}
			\foreach \i in {0,...,7} {
				\foreach \j in {0,...,7} {
					\pgfmathsetmacro{\x}{\i/\n}
					\pgfmathsetmacro{\y}{\j/\n}
					\draw[black, thin] (\x,\y) rectangle ++(\h,\h);
					\draw[black, thin] (\x,\y) -- ++(\h,\h);
				}
			}
			\draw[blue, thick] (0.5,0.5) circle[radius=0.3];
			\node[inner sep=0.8pt] at (0.2,0.85) {$\mathcal{T}_h^{N}$};
			\node[inner sep=0.8pt] at (0.75,0.75) {$\mathcal{T}_h^{\Gamma}$};
		\end{tikzpicture}
		\label{subfig:unfitted}
	\end{minipage}
	\hfill
	\begin{minipage}[b]{0.31\textwidth}
		\centering
		\textbf{(c)}\\[1ex]
		\begin{tikzpicture}[scale=3.5]
			\def\n{8}
			\pgfmathsetmacro{\h}{1/\n}
			\foreach \i in {0,...,7} {
				\foreach \j in {0,...,7} {
					\pgfmathsetmacro{\x}{\i/\n}
					\pgfmathsetmacro{\y}{\j/\n}
					\draw[black, thin] (\x,\y) rectangle ++(\h,\h);
					\draw[black, thin] (\x,\y) -- ++(\h,\h);
				}
			}
			\draw[red, thick]
			(0.8,0.5)  -- (0.75,0.65) -- (0.65,0.75) -- (0.5,0.8)  --
			(0.35,0.75) -- (0.25,0.65) -- (0.2,0.5)  -- (0.25,0.35) --
			(0.35,0.25) -- (0.5,0.2)   -- (0.65,0.25) -- (0.75,0.35) -- cycle;
			\node[red] at (0.5,0.85) {$\Gamma_h$};
		\end{tikzpicture}
		\label{subfig:fitted}
	\end{minipage}
	
	\caption{Geometric interface and mesh interaction:
		(a) the computational domain for the interface problem;
		(b) unfitted mesh~$\mathcal{T}_h$;
		(c) local anisotropic hybrid mesh~$\tilde{\mathcal{T}}_h$.}
	\label{fig:mesh}
\end{figure}	

We initiate the process by generating an interface-unfitted mesh $\mathcal{T}_h$, which serves as the background mesh (see Figure \ref{fig:mesh}(b)). By sequentially connecting the intersection points of the interface $\Gamma$ (blue line) and the mesh edges, a polygonal approximation $\Gamma_h$ (red line) of the interface $\Gamma$ is constructed. The resulting mesh, denoted by $\widetilde{\mathcal{T}}_h$ (see Figure \ref{fig:mesh}(c)), is an interface-fitted mesh that contains anisotropic triangles and quadrilaterals in the vicinity of the interface. The domain $\Omega$ is thereby partitioned into two polygonal subdomains $\Omega_{1,h}$ and $\Omega_{2,h}$ by $\Gamma_h$, which serve as approximations to $\Omega_1$ and $\Omega_2$, respectively.

Define the following mesh subsets:
\begin{align}
	&\mathcal{T}_h^{\Gamma} := \{K \in \mathcal{T}_h \,;\, K \cap \Gamma \neq \emptyset\},\\
	&\mathcal{T}_h^N := \mathcal{T}_h \setminus \mathcal{T}_h^{\Gamma}. 
\end{align}
Elements in $\mathcal{T}_h^{\Gamma}$ are referred to as interface elements. The mesh $\widetilde{\mathcal{T}}_h$ can be regarded as a refinement of $\mathcal{T}_h$. Let $\widetilde{\mathcal{E}}_h$ denote the set of all edges in $\widetilde{\mathcal{T}}_h$. Define $\widetilde{\mathcal{T}}_{h,i}$ as the subset of elements in $\widetilde{\mathcal{T}}_h$ that lie within $\Omega_{i,h}$. Let $\widetilde{\mathcal{E}}_h^{\Gamma}$ denote the collection of edges that coincide with $\Gamma_h$, and $\widetilde{\mathcal{E}}_h^N := \widetilde{\mathcal{E}}_h \setminus \widetilde{\mathcal{E}}_{h}^{\Gamma}$. Additionally, we denote the set of boundary edges by $\widetilde{\mathcal{E}}_h^0$.

\section{The nonconforming finite element methods}
\subsection{The $P_1$ nonconforming element space}
The $P_1$ nonconforming finite element space on the locally anisotropic hybrid mesh $\widetilde{\mathcal{T}}_h$ is defined by
\begin{equation}
	V_h = \left\{ v \in L^2(\Omega) \;\middle|\; v|_K \in P_1(K)\ \forall K \in \widetilde{\mathcal{T}}_h,\ \int_e [v]\,\mathrm{d}s = 0\ \forall e \in \widetilde{\mathcal{E}}_h \right\}.
\end{equation}
While this definition appears straightforward, further clarification is required to ensure that it is well-posed. Let $V_h(K)$ denote the restriction of $V_h$ to an element $K$. If $K \in \mathcal{T}_h^N$, then $V_h(K)$ corresponds to the standard Crouzeix–Raviart element. The nontrivial case is when $K \in \mathcal{T}_h^{\Gamma}$.

Each interface element $K$ is typically partitioned by the discrete interface $\Gamma_h$ into either two triangles or a triangle and a quadrilateral. Under reasonable assumptions, we restrict our attention to these two configurations. The following discussion concerns the construction and properties of basis functions on interface macro-elements, along with the associated interpolation error estimates.

\begin{figure}[ht]
	\centering
	\begin{tikzpicture}[scale=1]
		\coordinate (A1) at (0,0);
		\coordinate (A2) at (2.3,0);
		\coordinate (A4) at (1.2,2.3);
		\draw[thick] (A1) -- (A2) -- (A4) -- cycle;
		\node at (1,-0.3) {$K$};
		
		\coordinate (A3) at ($(A2)!0.3!(A4)$);
		\coordinate (A5) at ($(A1)!0.5!(A4)$);
		\draw[thin] (A3) -- (A5);
		
		\node[below left]  at (A1) {$A_1$};
		\node[below right] at (A2) {$A_2$};
		\node[above]       at (A4) {$A_4$};
		\node[right]       at (A3) {$A_3$};
		\node[left]        at (A5) {$A_5$};
		
		\node at (1.3,1.2) {$T$};
		\node at (1.3,0.4) {$Q$};
		
		\draw[->, thick](5.2,1)  -- node[above] {$F_K$} (3,1);
		
		\begin{scope}[xshift=6cm]
			\coordinate (Ah1) at (0,0);
			\coordinate (Ah2) at (2.3,0);
			\coordinate (Ah4) at (0,2.3);
			\draw[thick] (Ah1) -- (Ah2) -- (Ah4) -- cycle;
			\node at (1,-0.3) {$\hat K$};
			
			\coordinate (Ah3) at ($(Ah2)!0.3!(Ah4)$);
			\coordinate (Ah5) at ($(Ah1)!0.5!(Ah4)$);
			\draw[thin] (Ah3) -- (Ah5);
			
			\node[below left]  at (Ah1) {$\hat A_1$};
			\node[below right] at (Ah2) {$\hat A_2$};
			\node[above]       at (Ah4) {$\hat A_4$};
			\node[right]       at (Ah3) {$\hat A_3$};
			\node[left]        at (Ah5) {$\hat A_5$};
			
			\node at (0.5,1.4) {$\hat T$};
			\node at (0.6,0.4) {$\hat Q$};
		\end{scope}
	\end{tikzpicture}
	\caption{The interface macro element}
	\label{fig:macro-element}
\end{figure}
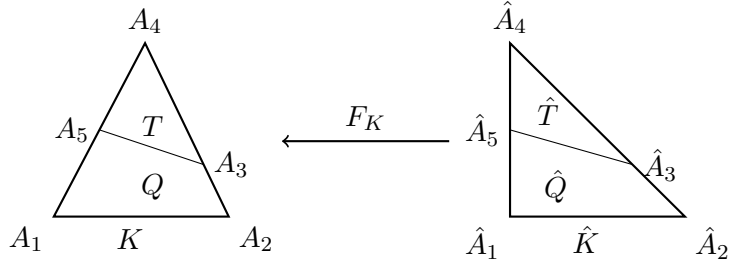
Consider a general interface macro-element $K$ as illustrated in Figure~\ref{fig:macro-element}. Define the cut ratio parameters by
\begin{equation*}
	t = \frac{|A_1A_5|}{|A_1A_4|}, \quad s =  \frac{|A_2A_3|}{|A_2A_4|}.
\end{equation*}
Clearly, $0 \leq s, t \leq 1$. Without loss of generality, we assume $t \geq s$; otherwise, we apply a reflection transformation to satisfy this condition.

\textbf{Case I.} The interface passes through a vertex of the macro-element, i.e., $t = 0$. In this case, the element is divided into two triangles. It can be easily shown that both sub-triangles satisfy the maximum angle condition; see \cite{Hu2021optimal}.

\textbf{Case II.} The interface intersects the interior of two edges of the macro-element, i.e., $0 < s \leq t < 1$. In this case, an affine mapping $F$ is used to map the physical macro-element $K$ to a reference element $\hat{K}$:
\begin{equation}
	\bm{x} = F(\hat{\bm{x}}) = \bm{B} \hat{\bm{x}} + \bm{b}.
\end{equation}
The coordinates of the points $\hat{A}_1, \hat{A}_2, \dots, \hat{A}_5$ are given by
\begin{equation*}
	(0, 0), \quad (1, 0), \quad (1 - s, s), \quad (0, 1), \quad (0, t),
\end{equation*}
where $0 < s \leq t < 1$. Note that when $t = 0$, $\hat{A}_5$ coincides with $\hat{A}_1$, which we exclude to avoid degeneracy. Since $\hat{A}_3$ and $\hat{A}_5$ cannot simultaneously coincide with the vertices of the triangle $\hat{A}_1\hat{A}_2\hat{A}_4$, we restrict $t$ to be strictly positive.
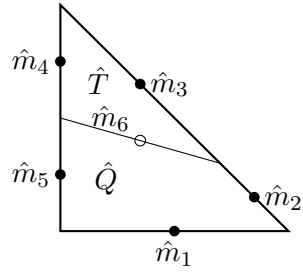
\begin{figure}[ht]
	\centering
	\begin{tikzpicture}[scale=1]
		\coordinate (Ah1) at (0,0);
		\coordinate (Ah2) at (3,0);
		\coordinate (Ah4) at (0,3);
		\draw[thick] (Ah1) -- (Ah2) -- (Ah4) -- cycle;
		
		\coordinate (Ah3) at ($(Ah2)!0.3!(Ah4)$);
		\coordinate (Ah5) at ($(Ah1)!0.5!(Ah4)$);
		\draw[thin] (Ah3) -- (Ah5);
		
		\coordinate (Mh1) at ($(Ah1)!0.5!(Ah2)$);
		\coordinate (Mh2) at ($(Ah2)!0.5!(Ah3)$);
		\coordinate (Mh3) at ($(Ah3)!0.5!(Ah4)$);
		\coordinate (Mh4) at ($(Ah4)!0.5!(Ah5)$);
		\coordinate (Mh5) at ($(Ah5)!0.5!(Ah1)$);
		\coordinate (Mh6) at ($(Ah3)!0.5!(Ah5)$);
		
		\foreach \M in {Mh1,Mh2,Mh3,Mh4,Mh5} {\fill (\M) circle (2pt);}
		\draw (Mh6) circle (2pt);
		
		\node[below]      at (Mh1) {$\hat m_1$};
		\node[right]      at (Mh2) {$\hat m_2$};
		\node[right]      at (Mh3) {$\hat m_3$};
		\node[left]       at (Mh4) {$\hat m_4$};
		\node[left]       at (Mh5) {$\hat m_5$};
		\node[above left] at (Mh6) {$\hat m_6$};
		
		\node at (0.5,2.0) {$\hat T$};
		\node at (0.6,0.7) {$\hat Q$};
	\end{tikzpicture}
	\caption{Degrees of freedom}
	\label{degree}
\end{figure}

We denote the edges $\hat{A}_1\hat{A}_2$, $\hat{A}_2\hat{A}_3$, $\hat{A}_3\hat{A}_4$, $\hat{A}_4\hat{A}_5$, $\hat{A}_5\hat{A}_1$, and $\hat{A}_3\hat{A}_5$ by $\hat{e}_1$ to $\hat{e}_6$, respectively, with $\hat{m}_i$ being the midpoint of $\hat{e}_i$ (see Figure \ref{degree}). The macro-element $\hat{K}$ consists of triangular ($\hat{T}$) and quadrilateral ($\hat{Q}$) subelements. 

The basis functions defined on the macro-element $\hat{K}$ resemble Hsieh–Clough–Tocher type element and are constructed according to the finite element triple $(\hat{K}, \mathcal{P}_{\hat{K}}, \mathcal{N})$, where
\begin{equation}
	\left\{
	\begin{aligned}
		&\mathcal{P}_{\hat{K}}=\{v\in L^2(\hat{K}); v|_{\hat{T}}\in P_1(\hat{T}),~v|_{\hat{Q}}\in P_1(\hat{Q}), [\![v(\hat{m}_6)]\!] =0\},	\\
		&\mathcal{N} = \{\mathcal{N}_1,\mathcal{N}_2,\cdots,\mathcal{N}_5\}~ \text{where} ~\mathcal{N}_i(v) = v(\hat{m}_i),~ 1\leq i\leq 5.
	\end{aligned}
	\right.
\end{equation}
This definition implies that there are no degrees of freedom on edges belonging to $\widetilde{\mathcal{E}}_h^{\Gamma}$. The unisolvence can be directly deduced via straightforward calculations. Furthermore, the explicit basis functions $\hat{\phi}_i \in \mathcal{P}_{\hat{K}}$ corresponding to $\mathcal{N}_i$ are given by:
\begin{align}
	\hat{\phi}_1 &= \begin{cases} 
		\dfrac{(s - t)\hat{x} + (s - 2)\hat{y} + t - st/2}{|\hat{Q}|}, & (\hat{x},\hat{y}) \in \hat{Q} \\[10pt]
		-\dfrac{(s - t)\hat{x} + (s - 1)\hat{y} + (1 + t - s - st)/2}{|\hat{T}|}, & (\hat{x},\hat{y}) \in \hat{T}
	\end{cases} \\
	\hat{\phi}_2 &= \begin{cases} 
		\dfrac{t\hat{x} + \hat{y} - t/2}{|\hat{Q}|}, & (\hat{x},\hat{y}) \in \hat{Q} \\[10pt]
		\dfrac{(s - t)\hat{x} + (s - 1)\hat{y} + (1 + t - s - st)/2}{|\hat{T}|}, & (\hat{x},\hat{y}) \in \hat{T}
	\end{cases} \\
	\hat{\phi}_3 &= \begin{cases} 
		0, & (\hat{x},\hat{y}) \in \hat{Q} \\[10pt]
		-\dfrac{(1 - s)\hat{x} + (1 - s)\hat{y} + (1 + t)(1 - s)/2}{|\hat{T}|}, & (\hat{x},\hat{y}) \in \hat{T}
	\end{cases} \\
	\hat{\phi}_4 &= \begin{cases} 
		0, & (\hat{x},\hat{y}) \in \hat{Q} \\[10pt]
		-\dfrac{(t - 1)\hat{x} + (1 - t)(1 - s)/2}{|\hat{T}|}, & (\hat{x},\hat{y}) \in \hat{T}
	\end{cases} \\
	\hat{\phi}_5 &= \begin{cases} 
		\dfrac{-s\hat{x} + (1 - s)\hat{y} + s/2}{|\hat{Q}|}, & (\hat{x},\hat{y}) \in \hat{Q} \\[10pt]
		-\dfrac{(s - t)\hat{x} + (s - 1)\hat{y} + (1 + t)(1 - s)/2}{|\hat{T}|}, & (\hat{x},\hat{y}) \in \hat{T}
	\end{cases}
\end{align}
where $|\hat{Q}| = \frac{1}{2}(t + s - st)$ and $|\hat{T}| = \frac{1}{2}(1 - t)(1 - s)$ denote the areas of the subelements.

\begin{lemma}\label{lem:basis-estimate}
	The basis functions satisfy the uniform estimate:
	\begin{equation}
		|\hat{\phi}_i|_{H^1(\hat{Q})} \lesssim |\hat{Q}|^{-1/2}, \quad 1 \leq i \leq 5.
	\end{equation}
\end{lemma}

\begin{proof}
	The bound follows from direct computation of the explicit expressions.
\end{proof}

Accordingly, the finite element space on the reference macro-element is defined as
\[
\hat{V}_h(\hat{K}) = \text{span}\left\{ \hat{\phi}_1, \hat{\phi}_2, \ldots, \hat{\phi}_5 \right\},
\]
and the corresponding physical-element space is given by
\[
V_h(K) = \left\{ v_h ~\middle|~ v_h = \hat{v}_h \circ F^{-1}_K,\; \hat{v}_h \in \hat{V}_h(\hat{K}) \right\}.
\]

\subsection{Weak formulation}
The continuous weak formulation of the elliptic interface problem is:
\begin{equation}\label{cweak}
	\begin{cases}
		\text{Find } u \in V := H^1_0(\Omega) \text{ such that} \\
		a(u,v) = F(v) \quad \forall v \in V,
	\end{cases}
\end{equation}
where $a(u,v) = (\beta\nabla u, \nabla v)$ and $F(v) = (f,v)$. The discrete variational formulation is given by:
\begin{equation}\label{dweak}
	\begin{cases}
		\text{Find } u_h \in \mathring{V}_h := \left\{ v \in V_h \ \middle|\ \int_e v  ds = 0\ \forall e \in \widetilde{\mathcal{E}}_h^{0} \right\} \text{ such that}\\
		a_h(u_h, v_h) = F(v_h) \quad \forall v_h \in \mathring{V}_h,
	\end{cases}
\end{equation}
where the discrete bilinear form is defined as $a_h(u_h, v_h) = \sum_{i=1}^{2} \sum_{K \in \widetilde{\mathcal{T}}_{h,i}} (\beta_i \nabla u_h, \nabla v_h)_K.$
This formulation corresponds to replacing the piecewise constant coefficient $\beta$ by its discrete approximation $\beta_h$, defined as
\[
\beta_h|_{K} =
\begin{cases}
	\beta_1 & K \in \widetilde{\mathcal{T}}_{h,1}, \\
	\beta_2 & K \in \widetilde{\mathcal{T}}_{h,2},
\end{cases}
\]
allowing the bilinear form to be equivalently expressed as
\[
a_h(u_h, v_h) = \sum_{K \in \widetilde{\mathcal{T}}_h} (\beta_h \nabla u_h, \nabla v_h)_K.
\]
\subsection{Interpolation error analysis}
Recall that $\widetilde{\mathcal{E}}_h^{\Gamma}$ denote the collection of all the edges align with $\Gamma_h$ and $\widetilde{\mathcal{E}}_h^N = \widetilde{\mathcal{E}}_h\setminus \widetilde{\mathcal{E}}_{h}^{\Gamma}$. Define $\pi_h:\tilde{H}^2(\Omega_{1}\cup\Omega_{2})\rightarrow \mathring{V}_h$ by
\begin{equation}
  \pi_h v(m_i) = (v(A_i)+v(A_{i+1}))/2\quad \forall e\in \mathcal{E}_h^N,
\end{equation}
where $A_i$ and $A_{i+1}$ is the two endpoints of $e$ and $m_i$ is the middle point. Denote $\pi_{K}$ the restriction of $\pi_{h}$ to element $K$.
\begin{remark} We note that for $K\in \widetilde{\mathcal{T}}_h^N$, $\pi_K v$ coincides with the standard nodal interpolation of linear conforming finite element. 
\end{remark}
We only need to analyze the interpolation error on interface macro element. For $K\in \mathcal{T}_h^{\Gamma}$, by a standard scaling argument, one has
	\begin{align}
		|v-\pi_{K}v|_{H^1(K)}
		&\leq C|\text{det}(B)|^{1/2}\|B^{-1}\||\hat{v}-\hat{\pi}_{\hat{K}}\hat{v}|_{H^1(\hat{K})}\\
		&\leq C|\hat{v}-\hat{\pi}_{\hat{K}}\hat{v}|_{H^1(\hat{K})}.
	\end{align}

For an arbitrary triangular element $T$, $\pi_{T}$ coincides with the standard nodal interpolation operator, i.e., $\pi_{T}v(A_i) = v(A_i)$. Since the triangular element $T$ satisfies the maximum angle condition, the following lemma holds, which is fundamentally attributed to Babuška \cite{Babuska1976on}:

\begin{lemma}
	Let $T$ be a triangular element satisfying the maximum angle condition. For all $v \in H^2(T)$, we have
	\begin{equation}
		|v - \pi_{T} v|_{H^1(T)} \lesssim h_T |v|_{H^2(T)},
	\end{equation}
	where $h_T$ denotes the diameter of element $T$.
\end{lemma}

\begin{figure}[ht]
	\centering
	\begin{tikzpicture}[scale=1]
		\coordinate (Ah1) at (0,0);
		\coordinate (Ah2) at (3.5,0);
		\coordinate (Ah4) at (0,3.5);
	
		\coordinate (Ah3) at ($(Ah2)!0.4!(Ah4)$);
		\coordinate (Ah5) at ($(Ah1)!0.5!(Ah4)$);
		\draw[dashed] (Ah2) -- (Ah5);
		\draw[dashed] (Ah1) -- (Ah3);
		\draw[thick] (Ah1) -- (Ah2) -- (Ah3) --(Ah5)--cycle;
		\coordinate (Mh1) at ($(Ah1)!0.5!(Ah2)$);
		\coordinate (Mh2) at ($(Ah2)!0.5!(Ah3)$);
		\coordinate (Mh3) at ($(Ah3)!0.5!(Ah4)$);
		\coordinate (Mh4) at ($(Ah4)!0.5!(Ah5)$);
		\coordinate (Mh5) at ($(Ah5)!0.5!(Ah1)$);
		\coordinate (Mh6) at ($(Ah3)!0.5!(Ah5)$);
		
		\foreach \M in {Mh1,Mh2,Mh5} {\fill (\M) circle (2pt);}
		\draw (Mh6) circle (2pt);
		\node[below]      at (Mh1) {$\hat m_1$};
		\node[right]      at (Mh2) {$\hat m_2$};
		\node[left]       at (Mh5) {$\hat m_5$};
		\node[above] at (Mh6) {$\hat m_6$};
		
		\node[below]      at (Ah1) {$\hat {A}_1$};
		\node[below]      at (Ah2) {$\hat {A}_2$};
		\node[right]      at (Ah3) {$\hat {A}_3$};
		\node[left]       at (Ah5) {$\hat {A}_5$};
		
	\end{tikzpicture}
	\caption{The reference quadrilateral element $\hat{Q}$}
	\label{split}
\end{figure}
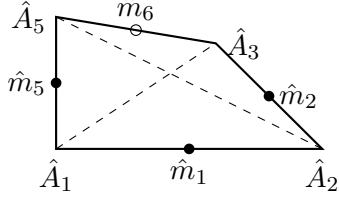
For elements $K \in \mathcal{T}_h^N$, the triangular subelements generated by the discrete interface $\Gamma_h$ satisfy the maximum angle condition. The interpolation error estimates on these triangles follow from the preceding lemmas. We now focus on the interpolation error analysis for quadrilateral subelements $Q \subset K\in \mathcal{T}_h^{\Gamma}$.

Let $\hat{Q}$ denote the quadrilateral subelement of a reference interface element $\hat{K}$ (see Figure \ref{split}), and let $\hat{T}_{ijk}$ represent the triangular subelement with vertices $\hat{A}_i$, $\hat{A}_j$, and $\hat{A}_k$. Define $\hat{\pi}^c$ as the standard linear conforming nodal interpolation operator on triangle $\hat{T}_{125}$.

The following trace theorem \cite{Acosta2000} provides geometrically explicit constants:
\begin{lemma}[Trace theorem]\label{lem:trace}
	For any triangle $T$ with diameter $h_T$ and edge $e$,
	\begin{equation}
		\|v\|_{L^2(e)} \leq \left( 2 \frac{|e|}{|T|} \right)^{1/2} \left( \|v\|_{L^2(T)} + h_T |v|_{H^1(T)} \right) \quad \forall v \in H^1(T).
	\end{equation}
\end{lemma}

The next interpolation estimate originates from Acosta and Durán \cite{Acosta2000}; we provide a concise proof for completeness.
\begin{lemma}[Acosta and Durán, 2000]\label{lem:acosta}
	For any $\hat{v} \in H^2(\hat{Q})$, we have  
	\begin{equation}
		|\hat{v} - \hat{\pi}^c \hat{v}|_{H^1(\hat{Q})} \lesssim |\hat{v}|_{H^2(\hat{Q})}.
	\end{equation}
\end{lemma}
\begin{proof}
	Let $\bm{w} = (w_1, w_2)^\top = \nabla (\hat{v} - \hat{\pi}^c \hat{v})$, it is sufficient to prove
	\begin{equation*}
		\|\bm{w}\|_{L^2(\hat{Q})}\lesssim |\bm{w}|_{H^1(\hat{Q})}.
	\end{equation*}
	Let $\bar{w}=(\bar{w}_1,\bar{w}_2)^t= \frac{1}{|\hat{Q}|}\int_{\hat{Q}} wds$, since $\int_{\hat{e}_1}w_1(x,0)ds=0$, we have
	\begin{align*}
	\|\bar{w}_1\|_{L^2(\hat{Q})} &=|\hat{Q}|^{1/2}|\bar{w}_1|\\
	&=|\hat{Q}|^{1/2}\big|\int_{\hat{e}_1} (w_1-\bar{w}_1)ds\big|\\
	&\lesssim |\hat{Q}|^{1/2}|\hat{e}_1|^{1/2}\|w_1-\bar{w}_1\|_{L^2(\hat{e}_1)}\\
	&\lesssim |\hat{Q}|^{1/2}|\hat{e}_1|^{1/2} \frac{|\hat{e}_1|^{1/2}}{|\hat{T}_{125}|^{1/2}}(\|w_1-\bar{w}_1\|_{L^2(\hat{T}_{125})}+|w_1|_{H^1(\hat{T}_{125})})\\
	&\lesssim \frac{|\hat{Q}|^{1/2}}{|T_{125}|^{1/2}}|\hat{e_1}||w_1|_{H^1(\hat{T}_{125})}\quad (\text{using } |\hat{Q}|\sim |T_{125}|)\\
	&\lesssim |w_1|_{H^1(\hat{Q})}.
   \end{align*}
	 Similarly, $\|\bar{w}_2\|_{L^2(\hat{Q})} \lesssim |w_2|_{H^1(\hat{Q})}$. Finally,
	\begin{align*}
		\|\bm{w}\|_{L^2(\hat{Q})} &\leq \|\bm{w} - \bar{\bm{w}}\|_{L^2(\hat{Q})} + \|\bar{\bm{w}}\|_{L^2(\hat{Q})} \\
		&\lesssim |\bm{w}|_{H^1(\hat{Q})}. 
	\end{align*}
\end{proof}

Applying Lemma \ref{lem:acosta}, we derive the following interpolation error estimate on $\hat{Q}$:

\begin{lemma}\label{lem:quad-interp}
	For any $\hat{v} \in H^2(\hat{Q})$, the following estimate holds:
	\begin{equation}
		|\hat{v} - \hat{\pi}_{\hat{Q}} \hat{v}|_{H^1(\hat{Q})} \lesssim |\hat{v}|_{H^2(\hat{Q})}.
	\end{equation}
\end{lemma}

\begin{proof}
	Using the definitions of $\hat{\pi}^c$ and $\hat{\pi}_{\hat{Q}}$, it follows that
	\begin{align*}
		\hat{\pi}^c \hat{v} - \hat{\pi}_{\hat{Q}} \hat{v} &= \sum_{i \in \{1,2,5\}} \left[ \left( \hat{\pi}^c \hat{v} \right)(\hat{m}_i) - \left( \hat{\pi}_{\hat{Q}} \hat{v} \right)(\hat{m}_i) \right] \hat{\phi}_i \\
		&= \left[ \left( \hat{\pi}^c \hat{v} \right)(\hat{m}_2) - \left( \hat{\pi}_{\hat{Q}} \hat{v} \right)(\hat{m}_2) \right] \hat{\phi}_2 \\
		&= \frac{1}{2} \left( \hat{\pi}^c \hat{v} - \hat{v} \right)(\hat{A}_3) \hat{\phi}_2.
	\end{align*}
	Since $\left( \hat{\pi}^c \hat{v} - \hat{v} \right)(\hat{A}_2) = 0$, we derive
	\begin{align*}
		\left( \hat{\pi}^c \hat{v} - \hat{v} \right)(\hat{A}_3) &= \left| \int_{\hat{e}_2} \partial_\tau \left( \hat{\pi}^c \hat{v} - \hat{v} \right) \, ds \right| \\
		&\leq |\hat{e}_2|^{1/2} \, \left| \hat{\pi}^c \hat{v} - \hat{v} \right|_{H^1(\hat{e}_2)} \\
		&\lesssim \frac{|\hat{e}_2|}{|\hat{T}_{123}|^{1/2}} \left( \left| \hat{\pi}^c \hat{v} - \hat{v} \right|_{H^1(\hat{T}_{123})} + |\hat{v}|_{H^2(\hat{T}_{123})} \right) \\
		&\lesssim \frac{|\hat{e}_2|}{|\hat{T}_{123}|^{1/2}} |\hat{v}|_{H^2(\hat{Q})} \\
		&\lesssim t^{1/2} |\hat{v}|_{H^2(\hat{Q})}.
	\end{align*}
	Using Lemma \ref{lem:basis-estimate}, we have	
	\begin{align*}
		|\hat{\phi}_2|_{H^1(\hat{Q})} \lesssim t^{-1/2}.
	\end{align*}
	Therefore, through the triangle inequality, we obtain
	\begin{align*}
		|\hat{v} - \hat{\pi}_{\hat{Q}} \hat{v}|_{H^1(\hat{Q})} &\leq |\hat{v} - \hat{\pi}^c \hat{v}|_{H^1(\hat{Q})} + |\hat{\pi}^c \hat{v} - \hat{\pi}_{\hat{Q}} \hat{v}|_{H^1(\hat{Q})} \\
		&\lesssim |\hat{v}|_{H^2(\hat{Q})} + \left| \left( \hat{\pi}^c \hat{v} - \hat{v} \right)(\hat{A}_3) \right| \, |\hat{\phi}_2|_{H^1(\hat{Q})} \\
		&\lesssim |\hat{v}|_{H^2(\hat{Q})}.
	\end{align*}
\end{proof}

Using a similar argument as Theorem 4.1 in \cite{Hu2021optimal}, we have
\begin{theorem} \label{pih}
For any $v\in H^2(\Omega_1\cup\Omega_2)$, we have
	\begin{align}
		|v-\pi_h v|_{H^1(\Omega)}\lesssim h|v|_{H^2(\Omega_{1}\cup\Omega_{2})}.
	\end{align}
\end{theorem}

\subsection{Consistency  Error Analysis}
In this section, we analyze the consistency error for the nonconforming finite element space. By definition, this error is expressed as
\begin{equation}\label{E_h}
	\begin{aligned}
		E_h(u,v_h) &= a_h(u,v_h) - (f,v_h) \\
		&= \sum_{i=1}^{2} \sum_{K \in \mathcal{T}_{h,i}} (\beta_i \nabla u, \nabla v_h)_K - \int_{\Omega} f v_h  dx \\
		&= \sum_{i=1}^{2} \sum_{K \in \mathcal{T}_{h,i}} \big( (\beta_{i+1} - \beta_i) \nabla u, \nabla v_h \big)_{K \cap \Omega_i} 
		+ \sum_{K} \sum_{e \subset \partial K} \int_{e} \beta \frac{\partial u}{\partial \boldsymbol{n}} v_h  ds \\
		&=: E_1(u,v_h) + E_2(u,v_h)
	\end{aligned}
\end{equation}
The consistency error for nonconforming elements in interface problems thus decomposes into two components: 
$E_1(u, v_h)$ arises from piecewise linear approximation of the curved interface, whereas $E_2(u, v_h)$ represents the intrinsic consistency error of the nonconforming method. 
The analysis of $E_1(u, v_h)$ is established in \cite{Hu2021optimal}. 

\begin{lemma}\label{E1}
	For any $u\in \tilde{H}^2(\Omega_1\cup\Omega_2)$, the following estimate holds:
	\begin{equation}
		E_1(u,v_h) \lesssim h \|u\|_{H^2(\Omega_1 \cup \Omega_2)} \|v_h\|_{V_h},
	\end{equation}
	where $\|\cdot\|_{V_h}$ denotes the piecewise $H^1$ semi-norm. 
\end{lemma}

This subsection focuses on rigorous analysis of $E_2(u, v_h)$. For notational brevity, define $\bm{\eta} = \beta \nabla u$ and $E(\bm{\eta}, v_h) := E_2(u, v_h)$. 
Although $\bm{\eta}$ possesses $H^1$-regularity in each subdomain $\Omega_i$, it may lack $H^1$-regularity over individual elements. 
To address this issue, we invoke the Sobolev extension theorem: there exist operators $T_i : H^1(\Omega_i) \to H^1(\Omega)$ satisfying
\begin{equation}
	T_i v_i\big|_{\Omega_i} = v_i \quad \text{and} \quad \|T_i v_i\|_{H^1(\Omega)} \lesssim \|v_i\|_{H^1(\Omega_i)}, \quad i=1,2.
\end{equation}
Define the piecewise extension operator
\begin{equation}
	T\bm{\eta} :=
	\begin{cases}
		T_1 \bm{\eta} & \text{in } \Omega_{1,h}, \\
		T_2 \bm{\eta} & \text{in } \Omega_{2,h}.
	\end{cases}
\end{equation}
Consequently, the consistency error decomposes as
\begin{align}
	E(\bm{\eta}, v_h) = E(T\bm{\eta}, v_h) + E(\bm{\eta} - T\bm{\eta}, v_h).
\end{align}
Prior to estimating the consistency error, we introduce notation that will be used throughout the subsequent analysis. For a function $v \in H^1(K)$, let $v^{(e)}$ denote the trace of $v$ on edge $e$. We define
\begin{align*}
	\bar{v}^{(e)} &:= \frac{1}{|e|} \int_{e} v  ds, \\
	\delta v^{(e)} &:= v^{(e)} - \bar{v}^{(e)}.
\end{align*}
For vector-valued functions $\bm{\eta}$, we denote its components by $\eta_x$ and $\eta_y$. When the superscript is replaced by an element $K$ (e.g., $\bar{v}^{(K)}$), the definitions carry analogous meanings over the element $K$. For an interface macro element $K$ (see Figure \ref{fig:macro-element}), we frequently analyze its subtriangular components. To enhance notational brevity, superscripts or subscripts may be omitted when no ambiguity arises.

The following refined trace theorem on anisotropic triangle elements (cf. [Dura2020]) accounts for element geometry.
\begin{lemma}
	Let $K$ be an arbitrary tirangle, for any edge $e \subset \partial K$, the following inequality holds:
	\begin{equation}\label{eq:aniso-trace}
		\|v\|_{L^{2}(e)} \leq \left( \frac{|e|}{|K|} \right)^{1/2} \left( \|v\|_{L^2(K)} + h_K |v|_{H^1(K)} \right).
	\end{equation}
\end{lemma}

The following lemma provides an estimate for $E(\bm{\eta} - T\bm{\eta}, v_h)$.
 
\begin{lemma}\label{lemma:eta-Teta}
	For any $\boldsymbol{\eta} \in \big( H^1(\Omega_1 \cup \Omega_2) \big)^2$, 
	\begin{equation}\label{eq:eta-diff}
		E(\boldsymbol{\eta} - T\boldsymbol{\eta}, v_h) \lesssim h \|\boldsymbol{\eta}\|_{H^1(\Omega_1 \cup \Omega_2)} \|v_h\|_{V_h}.
	\end{equation}
\end{lemma}

\begin{proof}
	Define $\bm{\omega} = (T_i\boldsymbol{\eta} - T_j\boldsymbol{\eta}) \cdot \bm{n}$. 
	Applying the trace theorem, Cauchy-Schwarz inequality, and standard interpolation estimates yields:
	\begin{equation*}
		\begin{aligned}
			E(\boldsymbol{\eta} - T\boldsymbol{\eta}, v_h)
			&= \sum_{e \in \mathcal{E}_h^{\Gamma}} \int_{e} \delta\omega^{(e)} \delta v_h^{(e)}  ds \\
			&\lesssim \sum_{T \in \mathcal{T}_h^{\Gamma}} \frac{|e|}{|T_k|} \left( \|\delta\omega^{(T_k)}\|_{0,T_k} + h_{T_k} |\delta\omega^{(T_k)}|_{1,T_k} \right)
			\left( \|\delta v_h^{(T_k)}\|_{0,T_k} + h_{T_k} |\delta v_h^{(T_k)}|_{1,T_k} \right) \\
			&\lesssim \left( \frac{|e|}{|T_k|} \right) h^2 \|\bm{\eta}\|_{H^1(\Omega_1 \cup \Omega_2)} \|v_h\|_{V_h}.
		\end{aligned}
	\end{equation*}
	
	Consider the interface position parameter $t \in (0,1)$ (see Figure \ref{fig:macro-element}) :
	\begin{itemize}
		\item For $0 < t \leq \frac{1}{2}$, select $T_k = T_{345}$ with area $|T_k| = \frac{1}{2}|e_3||e_4|\sin\theta_{345} \lesssim (1-t)|e_6|$. 
		This implies $\frac{|e|}{|T_k|} \lesssim (1-t)^{-1}h_T^{-1} \lesssim h_T^{-1}$.
		
		\item For $\frac{1}{2} < t < 1$, choose $T_k = T_{135}$ with area $|T_k| = \frac{1}{2}|e_5||e_6|\sin\theta_{135} \lesssim t h_T|e_6|$. 
		Consequently, $\frac{|e|}{|T_k|} \lesssim t^{-1}h_T^{-1} \lesssim h_T^{-1}$.
	\end{itemize}
	Combining both cases, we obtain $\frac{|e|}{|T_k|} \lesssim h_T^{-1}$, which completes the proof of \eqref{eq:eta-diff}.
\end{proof}

Since each quadrilateral element in the mesh $\widetilde{\mathcal{T}}_h$ can be subdivided along one diagonal into two triangular subelements satisfying the maximum angle condition, we denote the resulting refined mesh by $\widetilde{\mathcal{T}}_h^R$. The consistency error term $E(\bm{\eta}, v_h)$ can then be expressed as
\begin{align*}
	E(\bm{\eta}, v_h) = \sum_{K \in \widetilde{\mathcal{T}}_h^R} \sum_{e \subset \partial K} \int_e (\bm{\eta} \cdot \bm{n}) v_h  ds.
\end{align*}
The continuity of $v_h$ within quadrilateral elements implies that the contributions along the subdivided diagonals vanish. Consequently, it suffices to estimate the consistency error for triangular elements satisfying the maximum angle condition. These include both original elements in $\widetilde{\mathcal{T}}_h$ and subelements generated by diagonal subdivision of quadrilaterals.

Building upon the refined trace theorem, we establish two fundamental estimates:

\begin{lemma}\label{lemma:edge-estimates}
	Let $T$ be an arbitrary triangular element with diameter $h_T$, and $e$ an edge of $T$. Then:
	\begin{align}
		\|\delta v^{e}\|_{L^2(e)} &\lesssim \left( \frac{|e|}{|T|} \right)^{1/2} h_T |v|_{H^1(T)}. \label{eq:gen-est}
	\end{align}
	Moreover, if $v \in P_1(T)$, the following improved estimate holds:
	\begin{align}
		\|\delta v^{e}\|_{L^2(e)} &\lesssim \left( \frac{|e|}{|T|} \right)^{1/2} |e| |v|_{H^1(T)}. \label{eq:p1-est}
	\end{align}
\end{lemma}

\begin{proof}
	Utilizing the $L^2$-projection property of the integral mean and the trace theorem, we derive:
	\begin{equation*}
		\begin{aligned}
			\|\delta v^{e}\|_{L^2(e)} 
			&= \|v - \bar{v}^{e}\|_{L^2(e)} \\
			&\leq \|v - \bar{v}^{T}\|_{L^2(e)} \\
			&\lesssim \left( \frac{|e|}{|T|} \right)^{1/2} \left( \|v - \bar{v}^{T}\|_{L^2(T)} + h_T |v|_{H^1(T)} \right) \\
			&\lesssim \left( \frac{|e|}{|T|} \right)^{1/2} h_T |v|_{H^1(T)}.
		\end{aligned}
	\end{equation*}
	When $v \in P_1(T)$, the Poincaré inequality provides enhanced approximation:
	\begin{equation*}
		\begin{aligned}
			\|\delta v^{e}\|_{L^2(e)} 
			&\leq |e||v|_{H^1(e)} \\
			&\lesssim \left( \frac{|e|}{|T|} \right)^{1/2} |e|\left( |v |_{H^1(T)} + h_T|v|_{H^2(T)} \right) \\
			&= \left( \frac{|e|}{|T|} \right)^{1/2} |e| |v|_{H^1(T)}.
		\end{aligned}
	\end{equation*}
\end{proof}

The mesh $\tilde{\mathcal{T}}_h$ comprises both triangular and quadrilateral elements, where triangular elements satisfy the maximum angle condition, and quadrilateral elements satisfy the RDP condition. This property ensures that each quadrilateral can be subdivided along its longer diagonal into two triangles satisfying the maximum angle condition. We denote the resulting interface-refined mesh by $\tilde{\mathcal{T}}_h^r$, following the work \cite{Xu2016optimal} where $P_1$ conforming elements were employed on this mesh for interface problems. Given the continuity of the finite element space within quadrilateral elements, it suffices to estimate the consistency error on triangular elements $K\in \tilde{\mathcal{T}}_h^r$ satisfying the maximum angle condition.

As rotation and reflection transformations preserve integral values, we consider a representative triangle $K$ with vertices positioned at
\begin{align*}
	A_1(0,0), \quad A_2(h_1,0), \quad A_3(h_2,h_3).
\end{align*}
Denoting the interior angles at vertices $A_i$ by $\theta_i$, without loss of generality, we assume that $\theta_2 \leq \theta_1 \leq \theta_3$. The corresponding edges are parameterized as follows:
\begin{equation*}
	\begin{aligned}
		e_1 &= \left\{ (x,y) \mid y = y_1(x),\ h_2 \leq x \leq h_1 \right\} 
		= \left\{ (x,y) \mid x = x_1(y),\ 0 \leq y \leq h_3 \right\}, \\
		e_2 &= \left\{ (x,y) \mid y = y_2(x),\ 0 \leq x \leq h_2 \right\} 
		= \left\{ (x,y) \mid x = x_2(y),\ 0 \leq y \leq h_3 \right\}, \\
		e_3 &= \left\{ (x,y) \mid y = 0,\ 0 \leq x \leq h_1 \right\}.
	\end{aligned}
\end{equation*}
Based on the geometric properties of the mesh $\widetilde{\mathcal{T}}_h^r$, the following bounds hold:
\begin{equation}
	\begin{aligned}
		&1\lesssim \theta_1\leq \pi/2,\quad 0\leq \theta_2\leq \pi/2,\label{thetah}\\
		&h_2\leq h_3\leq h_1,\quad 0\leq h_2 \leq h_1/2.
	\end{aligned}
\end{equation}
The line integrals over the element edges can be transformed into coordinate integrals as follows (see Figure~\ref{triangleR} for the correspondence between edges and coordinate axes):
\begin{align*}
	&\text{For edge } e_1: && \int_{e_1} v^{(1)} \sin\theta_2 \, ds = \int_{0}^{h_3} v^{(1)} \, dy, \quad \int_{e_1} v^{(1)} \cos\theta_2 \, ds = \int_{h_2}^{h_1} v^{(1)} \, dx, \\
	&\text{For edge } e_2: && \int_{e_2} v^{(2)} \sin\theta_1 \, ds = \int_{0}^{h_3} v^{(2)} \, dy, \quad \int_{e_2} v^{(2)} \cos\theta_1 \, ds = \int_{0}^{h_2} v^{(2)} \, dx, \\
	&\text{For edge } e_3: && \int_{e_3} v^{(3)} \, ds = \int_{0}^{h_1} v^{(3)} \, dx.
\end{align*}
 Here, $v^{(i)}$ denotes the restriction of $v$ to the edge $e_i$.

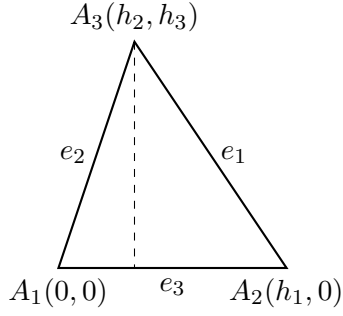
\begin{figure}[ht]
	\centering
	\begin{tikzpicture}[scale=1]
		\coordinate (Ah1) at (0,0);
		\coordinate (Ah2) at (3,0);
		\coordinate (Ah3) at (1,3);
		\coordinate (Ah4) at (1,0);
		\coordinate (Mh3) at ($(Ah1)!0.5!(Ah2)$);
		\coordinate (Mh1) at ($(Ah2)!0.5!(Ah3)$);
		\coordinate (Mh2) at ($(Ah3)!0.5!(Ah1)$);
		\draw[thick] (Ah1) -- (Ah2) -- (Ah3) -- cycle;
		\draw[dashed] (Ah3) -- (Ah4);
		
		\node at ($(Ah1) + (35:0.6cm)$) {$\theta_1$};  
		\node at ($(Ah2) + (150:0.6cm)$) {$\theta_2$}; 
		
		\node[below]      at (Ah1) {$A_1(0,0)$};
		\node[below]      at (Ah2) {$A_2(h_1,0)$};
		\node[above]      at (Ah3) {$A_3(h_2,h_3)$};
		\node[right]       at (Mh1) {$e_1$};
		\node[left]      at (Mh2) {$e_2$};
		\node[below]     at (Mh3) {$e_3$}; 
	\end{tikzpicture}
	\caption{triangle element satisfying Maxac.}
	\label{triangleR}
\end{figure}

\begin{lemma}\label{lemma:Eeta}
	For any $\boldsymbol{\eta}\in (H^1(\Omega_1\cup\Omega_2))^2$ and $v_h \in \mathring{V}_h$, it holds
	\begin{equation}\label{eq:Eeta-bound}
		E(T\boldsymbol{\eta},v_h) \lesssim h \|\boldsymbol{\eta}\|_{H^1(\Omega_1 \cup \Omega_2)} \|v_h\|_{V_h}.
	\end{equation}
\end{lemma}

\begin{proof}
	We decompose the consistency error element-wise:
	\begin{equation*}
		\begin{aligned}
			E(T\boldsymbol{\eta},v_h) 
			&= \sum_{e\in\mathcal{E}_h} \int_e \{T\bm{\eta}\cdot\boldsymbol{n}\}[v_h]  ds \\
			&= \sum_{e\in\mathcal{E}_h} \int_e \{ (T\bm{\eta} - \overline{T\bm{\eta}}) \cdot \boldsymbol{n} \} [v_h - \bar{v}_h]  ds \\
			&= \sum_{K\in\mathcal{T}_h^r} \sum_{e\subset\partial K} \int_e ( (T\bm{\eta} - \overline{T\bm{\eta}}) \cdot \boldsymbol{n} ) (v_h - \bar{v}_h)  ds \\
			&= \sum_{K\in\mathcal{T}_h^r} E_K(T\bm{\eta},v_h),
		\end{aligned}
	\end{equation*}
	where $E_K(T\bm{\eta},v_h) = \sum_{e\subset\partial K} \int_e (\delta T\bm{\eta} \cdot \boldsymbol{n}) \delta v_h  ds$. 
	The $H^1$-regularity of $T\bm{\eta}$ on each $K$ is ensured by the extension operator. 
	For notational brevity when unambiguous, we denote $T\bm{\eta}$ simply by $\bm{\eta}$.
	
	\begin{align*}
		E_K(\bm{\eta},v_h) 
		&= \left( -\int_0^{h_3} \delta \eta_x^{(1)} \delta v_h^{(1)}  dy + \int_0^{h_3} \delta \eta_x^{(2)} \delta v_h^{(2)}  dy \right) \\
		&\quad + \left( \int_0^{h_2} \delta \eta_y^{(2)} \delta v_h^{(2)}  dx + \int_{h_2}^{h_1} \delta \eta_y^{(1)} \delta v_h^{(1)}  dx 
		- \int_{0}^{h_1} \delta \eta_y^{(3)} \delta v_h^{(3)}  dx \right) \\
		&=: (I) + (II).
	\end{align*}
	Utilizing $\int_{e_i} \delta v_h^{(i)}  ds = \int_{e_i} \delta \eta_x^{(i)}  ds = \int_{e_i} \delta \eta_y^{(i)}  ds = 0$, we introduce intermediate terms.
	
	\textbf{Estimation of $(I)$:} Applying the Cauchy-Schwarz inequality, Lemma \ref{lemma:edge-estimates}, and geometric properties \eqref{thetah}:
	\begin{align*}
		(I) &= \int_0^{h_3} \left( \delta\eta_x^{(2)} - \delta\eta_x^{(1)} \right) \delta v_h^{(1)}  dy 
		+ \int_0^{h_3} \delta\eta_x^{(2)} \left( \delta v_h^{(2)} - \delta v_h^{(1)} \right)  dy \\
		&= \int_0^{h_3} \left( \int_{x_1(y)}^{x_3(y)} \partial_x \eta_x  dx \right) \delta v_h^{(1)}  dy
		+ \int_0^{h_3} \delta\eta_x^{(2)} \left( \int_{x_1(y)}^{x_3(y)} \partial_x v_h  dx \right)  dy \\
		&\leq |\sin\theta_2|^{1/2}h_1^{1/2}\|\partial_x\eta_x\|_{L^2(K)}\|\delta v_h^{(1)}\|_{L^2(e_1)} + |\sin\theta_1|^{1/2}h_1^{1/2}\|\delta \eta_x^{(2)}\|_{L^2(e_2)}\|\partial_x v_h\|_{L^2(K)} \\
		&\lesssim |e_1|^{1/2}|\eta_x|_{H^1(K)}|v_h|_{H^1(K)} + |e_3|^{1/2}|\eta_x|_{H^1(K)}|v_h|_{H^1(K)} \\
		&\lesssim h_1|\eta_x|_{H^1(K)}|v_h|_{H^1(K)}.
	\end{align*}
	
	\textbf{Estimation of $(II)$:} After introducing intermediate terms, we decompose:
	\begin{align*}
		(II) &= \int_0^{h_2} (\delta\eta_y^{(2)} - \delta\eta_y^{(3)}) \delta v_h^{(2)}  dx 
		+ \int_{h_2}^{h_1} (\delta \eta_y^{(1)} - \delta \eta_y^{(3)}) \delta v_h^{(1)}  dx \\
		&\quad + \int_{0}^{h_1} \delta \eta_y^{(3)} (v_h^{(2)} - v_h^{(3)})  dx 
		+ \int_{h_2}^{h_1} \delta \eta_y^{(3)} (v_h^{(1)} - v_h^{(3)})  dx \\
		&\quad + \int_{0}^{h_2} \delta \eta_y^{(3)} \bar{v}_h^{(2)}  dx 
		+ \int_{h_2}^{h_1} \delta \eta_y^{(3)} \bar{v}_h^{(1)}  dx \\
		&= (II)_a + (II)_b + (II)_c.
	\end{align*}
	Terms $(II)_a$ and $(II)_b$ are bounded similarly to $(I)$. For $(II)_c$, using geometric properties \eqref{thetah}:
	\begin{align*}
		(II)_c &= (\bar{v}_h^{(2)} - \bar{v}_h^{(1)}) \int_{0}^{h_2} \delta \eta_y^{(3)}  dx \\
		&= h_3^{-1} \left( \int_0^{h_3} v_h^{(2)}  dy - \int_0^{h_3} v_h^{(1)}  dy \right) \int_0^{h_2} \delta \eta_y^{(3)}  dx \\
		&= h_3^{-1} \left( \int_0^{h_3} \int_{x_3(y)}^{x_1(y)} \partial_x v_h  dx dy \right) \int_0^{h_2} \delta \eta_y^{(3)}  dx \\
		&\leq h_3^{-1}h_2^{1/2} |K|^{1/2}\|\partial_x v_h\|_{L^2(K)} \|\delta \eta_y^{(3)}\|_{L^2(e_3)} \\
		&\lesssim h_1 |\eta_y|_{H^1(K)} |v_h|_{H^1(K_{3})}.
	\end{align*}
	Here $K_3$ is a shape-regular triangular element with edge $e_3$ and diameter $\mathcal{O}(h_1)$, which does not necessary belong to the mesh $\tilde{\mathcal{T}}_h^r$. Although adjacent $K_3$ elements may overlap, the overlap is bounded.
	
	Combining estimates for $(I)$, $(II)_a$, $(II)_b$, and $(II)_c$, then summing over $K \in \mathcal{T}_h^r$, we have
	\begin{align*}
		E(T\boldsymbol{\eta},v_h) &\lesssim h \left( \|T_1\boldsymbol{\eta}\|_{H^1(\Omega)} + \|T_2\boldsymbol{\eta}\|_{H^1(\Omega)} \right) \|v_h\|_{V_h} \\
		&\lesssim h \|\boldsymbol{\eta}\|_{H^1(\Omega_1\cup\Omega_2)} \|v_h\|_{V_h}.
	\end{align*}
	The proof completes.
\end{proof}
Combining Lemmas \ref{E1}, \ref{lemma:eta-Teta} and \ref{lemma:Eeta}, we establish the following fundamental estimate for the consistency error:

\begin{theorem}\label{thm:consistency}
	For any $u \in \tilde{H}^2(\Omega_1 \cup \Omega_2)$, the consistency error satisfies
	\begin{equation}
		E_h(u, v_h) \lesssim h \|u\|_{H^2(\Omega_1 \cup \Omega_2)} \|v_h\|_{V_h}.
	\end{equation}
\end{theorem}

\subsection{Error estimates for the interface problems}
With the consistency error bound established in Theorem \ref{thm:consistency} and the interpolation estimates from previous sections, we now derive the main convergence result for the interface problem:

\begin{theorem}\label{thm:main-error}
	Let $u \in \tilde{H}^2(\Omega_1 \cup \Omega_2)$ and $u_h$ be solutions of problems $(P)$ and $(P_h)$, respectively. The following optimal-order error estimate holds:
	\begin{align}
		\|u - u_h\|_{V_h} \lesssim h \|u\|_{H^2(\Omega_1 \cup \Omega_2)}.
	\end{align}
\end{theorem}

\begin{proof}
	Applying Strang's second lemma for nonconforming methods:
	\begin{align*}
		\|u - u_h\|_{V_h} &\lesssim \inf_{w_h \in V_h} \|u - w_h\|_{V_h} + \sup_{v_h \in V_h \backslash \{0\}} \frac{|E_h(u, v_h)|}{\|v_h\|_{V_h}} \\
		&\lesssim \|u - \pi_h u\|_{V_h} + \sup_{v_h \in V_h} \frac{|E_h(u, v_h)|}{\|v_h\|_{V_h}} \\
		&\lesssim h \|u\|_{H^2(\Omega_1 \cup \Omega_2)} \quad \text{(by Theorems \ref{pih} and \ref{thm:consistency})}.
	\end{align*}
\end{proof}

\section{Numerical examples}
In this section, we present several numerical examples to validate the theoretical results of our nonconforming element method. The validation is primarily conducted from three perspectives:
\begin{itemize}
	\item Approximation Capability for Low-Regularity Solutions: By varying the discontinuous coefficient (specifically, testing with coefficient jumps $\beta_1/\beta_2$ ranging from $10^4$ to $10^{-4}$), we examine the finite element method's ability to approximate solutions with low regularity.
	
	\item Interface Geometry Robustness: We verify the method’s robustness with respect to interface geometry, investigating its performance across interfaces of different shapes.
	
	\item Robustness to Interface and Element Cut Positions: We examine the robustness of the method with respect to the relative positions of the interface and element cuts.
\end{itemize}

\subsection{Example 1: Circular Interface}
To first validate the approximation capability for low-regularity solutions under a simple benchmark geometry, we consider a circular interface in Example 1. Circular interfaces are symmetric, widely used as a baseline for interface problems, and allow clear isolation of the impact of discontinuous coefficients on solution regularity. The computational domain is $\Omega$ is $(-1,1)\times(-1,1).$ The interface is a circle centered at the origin with radius $r=0.6,$ i.e.,
\begin{equation*}
	\varPhi_{\Gamma}(x,y) = x^2+y^2-(0.6)^2.
\end{equation*}
The exact solution is chosen as follows
\begin{equation*}
	u = \frac{1}{\beta}\varPhi_{\Gamma}(x,y)(x^2-1)(y^2-1).
\end{equation*}
Numerical tests are performed for viscosity jump ratios $\beta1/\beta2=10^4,10^2,1,10^{-2}$ and $10^{-4}.$ Numerical results are shown in Tables \ref{e1:1}-\ref{e1:5}.
\begin{figure}[H]
	\vspace{-10pt}
	\begin{minipage}{0.45\linewidth}        	
		\centering
		\includegraphics[scale=0.4]{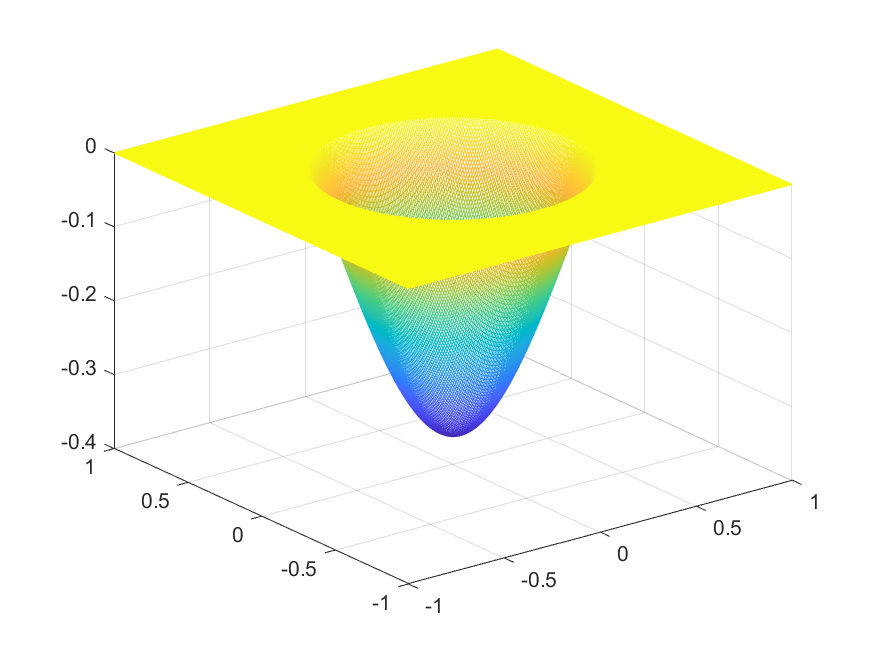}      	
	\end{minipage}
	\begin{minipage}{0.45\linewidth}
		\centering
		\includegraphics[scale=0.4]{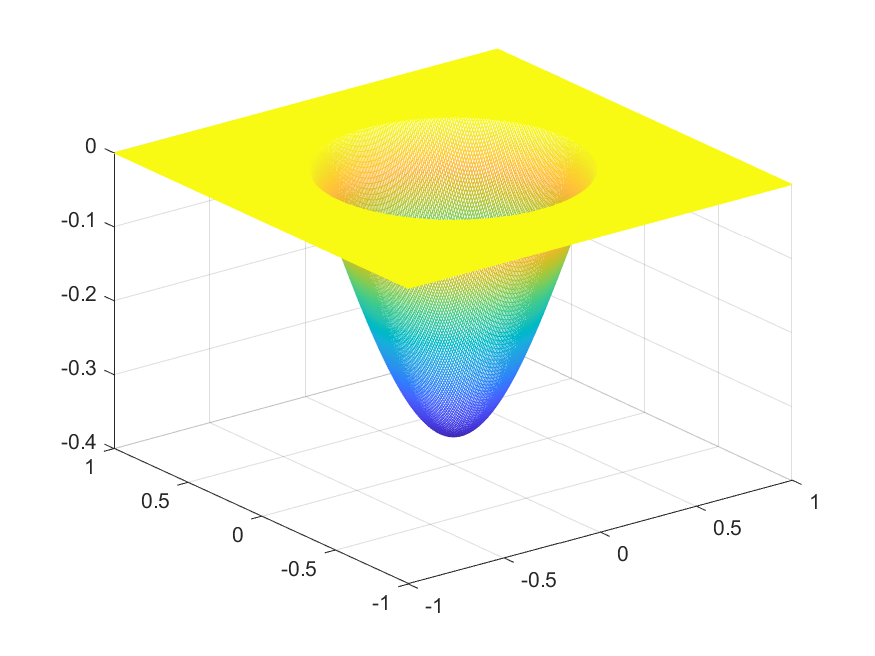}
	\end{minipage}
	\vspace{-10pt}
	\caption{The true solution $u$ (left) and the numerical solution $u_h$ (right) in \textbf{Example 1} with $\beta_1=10000,\,\beta_2=1$. \label{e1:u1}
	} 
\end{figure}  

\begin{figure}[H]
	\vspace{-10pt}
	\begin{minipage}{0.45\linewidth}        	
		\centering
		\includegraphics[scale=0.4]{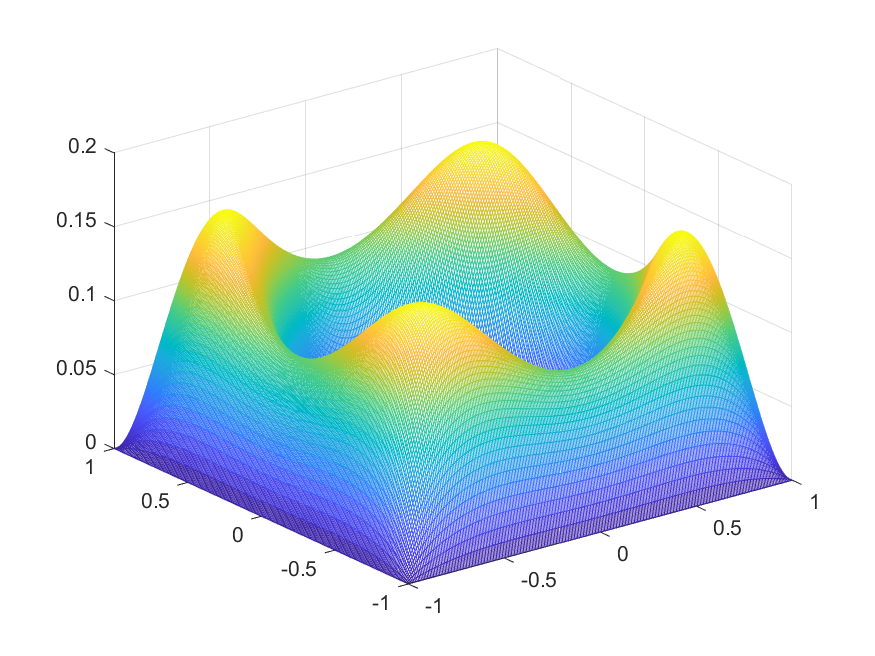}      	
	\end{minipage}
	\begin{minipage}{0.45\linewidth}
		\centering
		\includegraphics[scale=0.4]{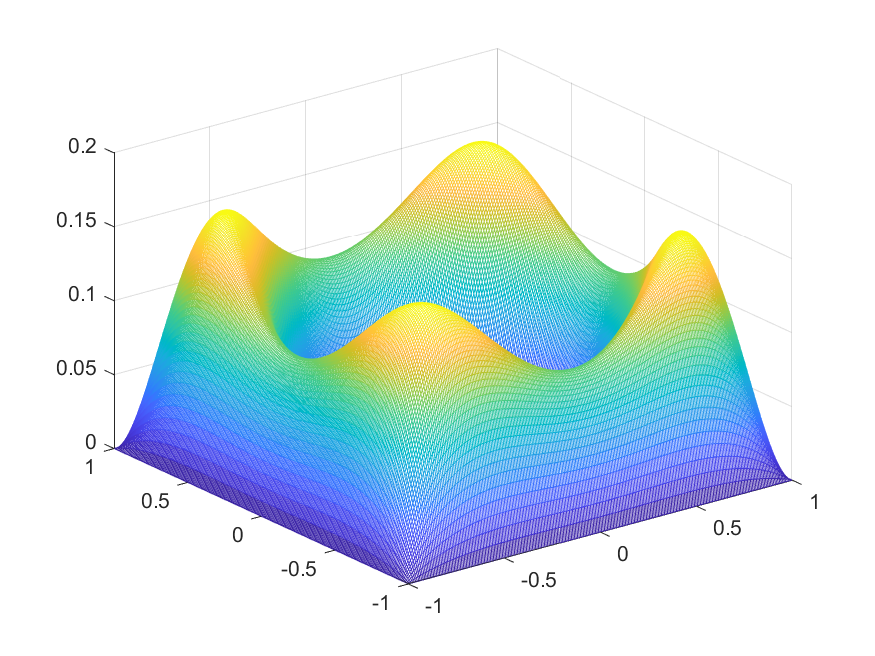}
	\end{minipage}
	\vspace{-10pt}
	\caption{The true solution $u$ (left) and the numerical solution $u_h$ (right) in \textbf{Example 1} with $\beta_1=1,\,\beta_2=10000$. \label{e1:u2}
	}
\end{figure}

\begin{table}[H]
	\caption{Numerical results for \textbf{Example 1} with $\beta_1 =10000,\,\beta_2=1$. } 
	\centering
	\begin{tabular}{c c c c c}
		\toprule
		\text{$\frac{1}{h}$}&
		\text{$\|u-u_h\|_{L^2(\Omega)}$}&
		\text{order}&
		\text{$|u-u_h|_{H^1(\Omega)}$}&
		\text{order}\\
		\midrule
		16   & 6.5063e-4 &  & 4.1746e-2 &  \\
		32   & 1.8759e-4 & 1.7943 & 2.1210e-2 & 0.9769 \\
		64   & 4.4875e-5 & 2.0636 & 1.0670e-2 & 0.9912 \\
		128  & 1.1139e-5 & 2.0103 & 5.3511e-3 & 0.9956 \\
		256  & 2.8806e-6 & 1.9512 & 2.6839e-3 & 0.9955 \\
		\bottomrule
	\end{tabular}\label{e1:1}
\end{table}
\begin{table}[H]
	\caption{Numerical results for \textbf{Example 1} with $\beta_1 =100,\,\beta_2=1$. } 
	\centering
	\begin{tabular}{c c c c c}
		\toprule
		\text{$\frac{1}{h}$}&
		\text{$\|u-u_h\|_{L^2(\Omega)}$}&
		\text{order}&
		\text{$|u-u_h|_{H^1(\Omega)}$}&
		\text{order}\\
		\midrule
		16 & 6.4573e-4  &  & 4.1777e-2 &  \\
		32 &  1.8587e-4 & 1.7966 & 2.1221e-2 & 0.9772 \\
		64 &  4.4507e-5 & 2.0622 & 1.0677e-2 & 0.9910 \\
		128 & 1.1053e-5 & 2.0095 & 5.3549e-3 & 0.9956 \\
		256 & 2.8220e-6 & 1.9697 & 2.6834e-3 & 0.9968 \\	
		\bottomrule
	\end{tabular}\label{e1:2}
\end{table}
\begin{table}[H]
	\caption{Numerical results for \textbf{Example 1} with $\beta_1 =1,\,\beta_2=1$. } 
	\centering
	\begin{tabular}{c c c c c }
		\toprule
		\text{$\frac{1}{h}$}&
		\text{$\|u-u_h\|_{L^2(\Omega)}$}&
		\text{order}&
		\text{$|u-u_h|_{H^1(\Omega)}$}&
		\text{order}\\
		\midrule
		16 &  3.0198e-3 &  & 1.6998e-1 &  \\
		32 &  7.6008e-4 & 1.9902 & 8.5373e-2 & 0.9935 \\
		64 &  1.9056e-4 & 1.9959 & 4.2742e-2 & 0.9981 \\
		128 &  4.7698e-5 & 1.9983 & 2.1383e-2 & 0.9991 \\
		256 &  1.1931e-5 & 1.9992 & 1.0695e-2 & 0.9996 \\
		\bottomrule
	\end{tabular}\label{e1:3}
\end{table}
\begin{table}[H]
	\caption{Numerical results for \textbf{Example 1} with $\beta_1 =1,\,\beta_2=100$. } 
	\centering
	\begin{tabular}{c c c c c }
		\toprule
		\text{$\frac{1}{h}$}&
		\text{$\|u-u_h\|_{L^2(\Omega)}$}&
		\text{order}&
		\text{$|u-u_h|_{H^1(\Omega)}$}&
		\text{order}\\
		\midrule
		16 & 3.0692e-3 &  & 1.6473e-1 &  \\
		32 &  7.8090e-4 & 1.9746 & 8.2682e-2 & 0.9944 \\
		64 &  1.9477e-4 & 2.0034 & 4.1385e-2 & 0.9984 \\
		128 &  4.8704e-5 & 1.9996 & 2.0702e-2 & 0.9993 \\
		256 &  1.2200e-5 & 1.9972 & 1.0353e-2 & 0.9998 \\
		\bottomrule
	\end{tabular}\label{e1:4}
\end{table}
\begin{table}[H]
	\caption{Numerical results for \textbf{Example 1} with $\beta_1 =1,\,\beta_2=10000$. } 
	\centering
	\begin{tabular}{c c c c c }
		\toprule
		\text{$\frac{1}{h}$}&
		\text{$\|u-u_h\|_{L^2(\Omega)}$}&
		\text{order}&
		\text{$|u-u_h|_{H^1(\Omega)}$}&
		\text{order}\\
		\midrule
		16 & 3.0705e-3 &  & 1.6473e-1 &  \\
		32 &  7.8130e-4 & 1.9745 & 8.2683e-2 & 0.9944 \\
		64 &  1.9486e-4 & 2.0034 & 4.1385e-2 & 0.9985 \\
		128 &  4.8728e-5 & 1.9996 & 2.0702e-2 & 0.9993 \\
		256 &  1.2150e-5 & 2.0037 & 1.0355e-2 & 0.9995 \\
		\bottomrule
	\end{tabular}\label{e1:5}
\end{table}

\subsection{Example 2: Flower-shaped Interface}
%
Example 1 illustrates the method’s performance for smooth, symmetric circular interfaces, but practical problems often involve irregular geometries with varying curvatures—features that heighten numerical challenges. To test geometric robustness, Example 2 uses a flower-shaped interface: asymmetric, non-convex, and defined by alternating curvatures, serving as a stricter benchmark for adaptability. The computational domain remains \(\Omega = (-1,1) \times (-1,1)\). The interface is described by the level set function:  
\begin{equation*}
 	\varPhi_{\Gamma}(x,y) = \sqrt{x^2 + y^2} - \frac{1}{2} + 2^{\sin\left(5\arctan\frac{y}{x}\right) - 3},
\end{equation*}  
where the interface geometry is visualized in Figure \ref{fig:flowerinterface}.  
The exact solution is chosen as:  
\begin{equation*}
 	u = \frac{1}{\beta}\varPhi_{\Gamma}(x,y)\sin(\pi x)\sin(\pi y),
\end{equation*}  
consistent in form with Example 1 to isolate geometric effects—here, the \(\sin(\pi x)\sin(\pi y)\) term introduces additional spatial variation, testing the method’s handling of combined geometric and solution complexity. Numerical results validating convergence and robustness are provided in Tables \ref{e2:1}–\ref{e2:5}.
 \begin{figure}[H]     	
	\centering
	\includegraphics[scale=0.5]{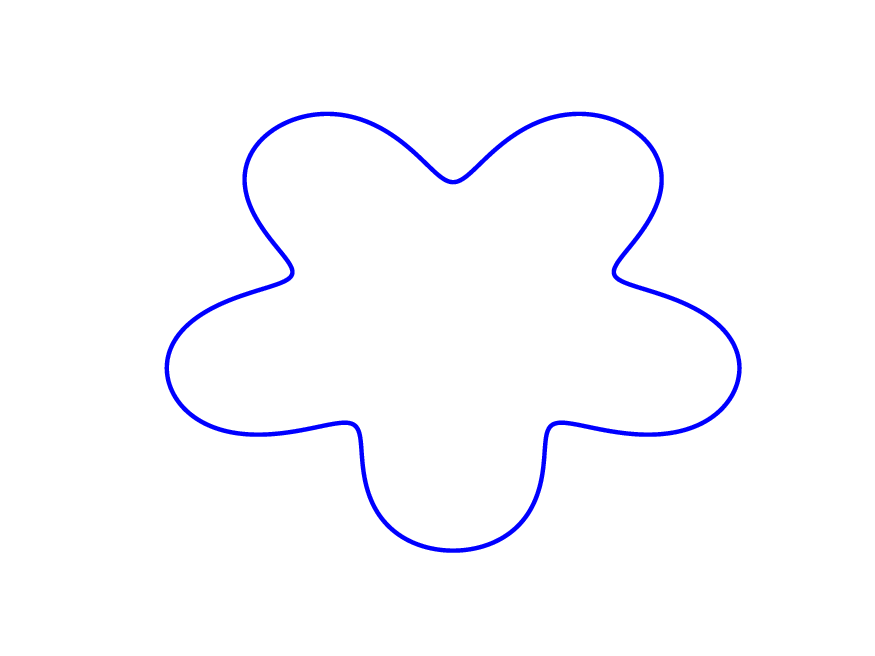}
	\caption{Flower-shaped interface defined by \(\varPhi_{\Gamma}(x,y) = 0\).}
	\label{fig:flowerinterface}
\end{figure}  
\begin{figure}[H]
	\vspace{-10pt}
	\begin{minipage}{0.5\linewidth}        	
		\centering
		\includegraphics[scale=0.5]{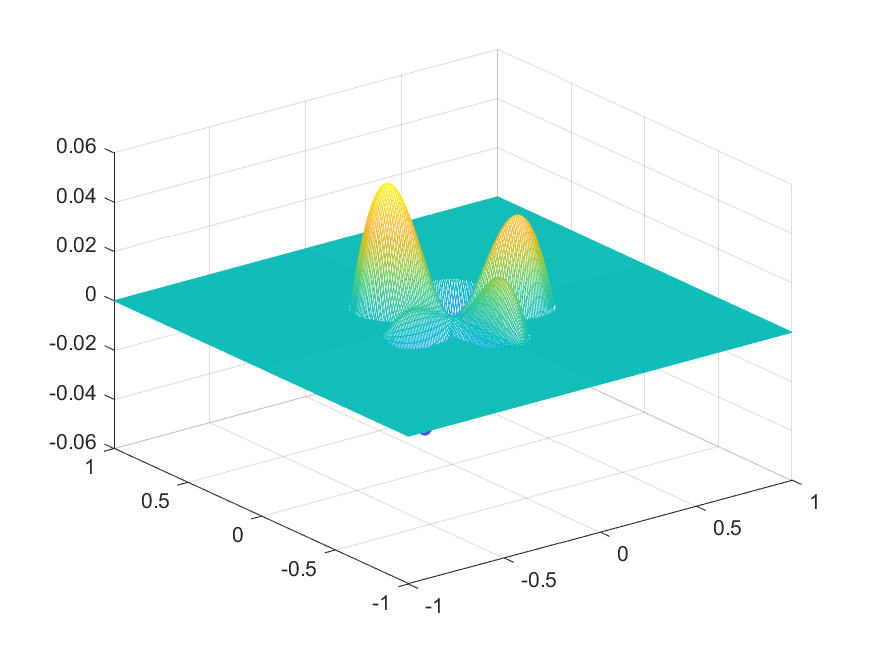}      	
	\end{minipage}
	\begin{minipage}{0.5\linewidth}
		\centering
		\includegraphics[scale=0.5]{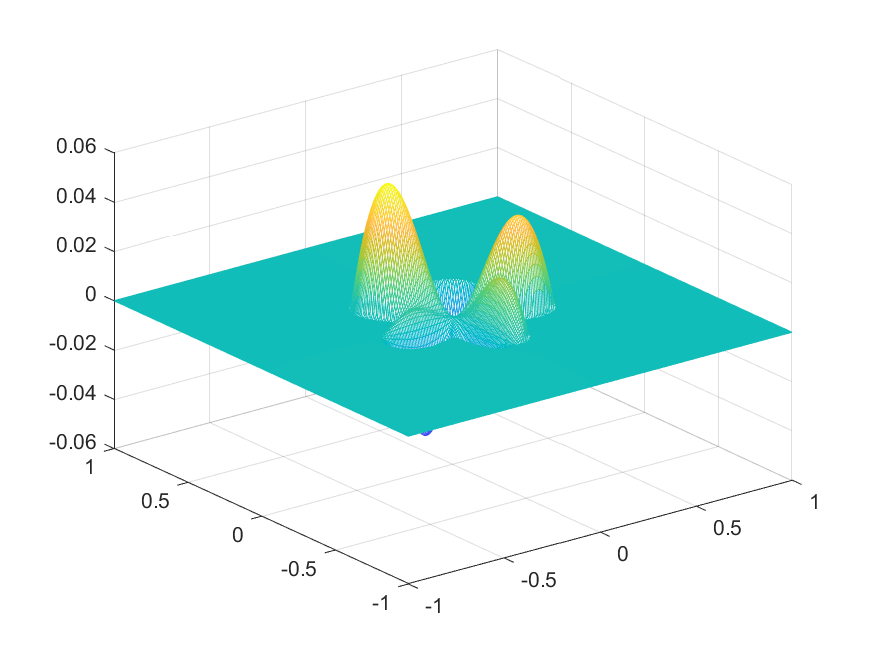}
	\end{minipage}
	\vspace{-10pt}
	\caption{The true solution $u$ (left) and the numerical solution $u_h$ (right) in \textbf{Example 2} with $\beta_1=10000,\,\beta_2=1$. 
	}
\end{figure}
\begin{figure}[H]
	\vspace{-10pt}
	\begin{minipage}{0.5\linewidth}        	
		\centering
		\includegraphics[scale=0.5]{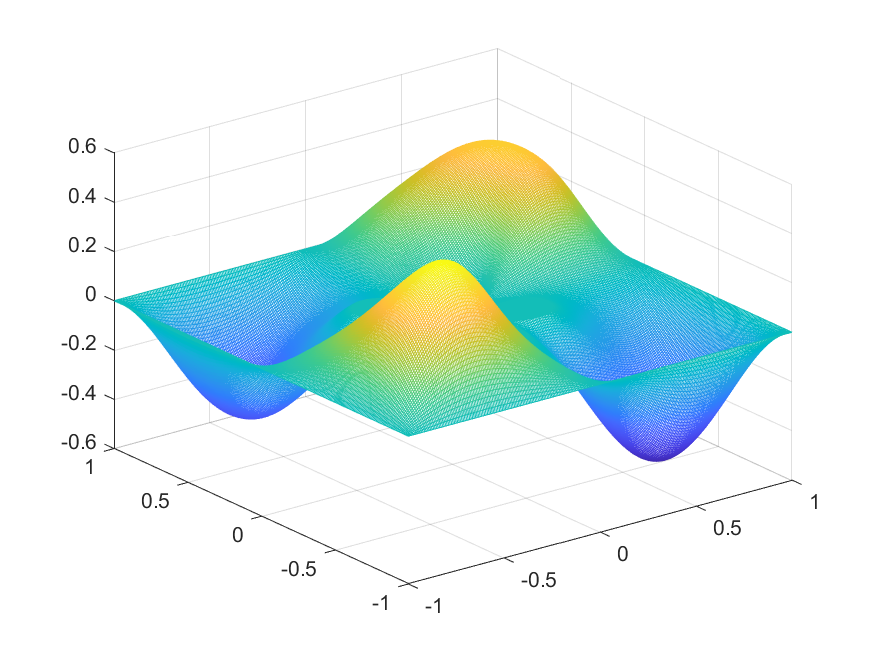}      	
	\end{minipage}
	\begin{minipage}{0.5\linewidth}
		\centering
		\includegraphics[scale=0.5]{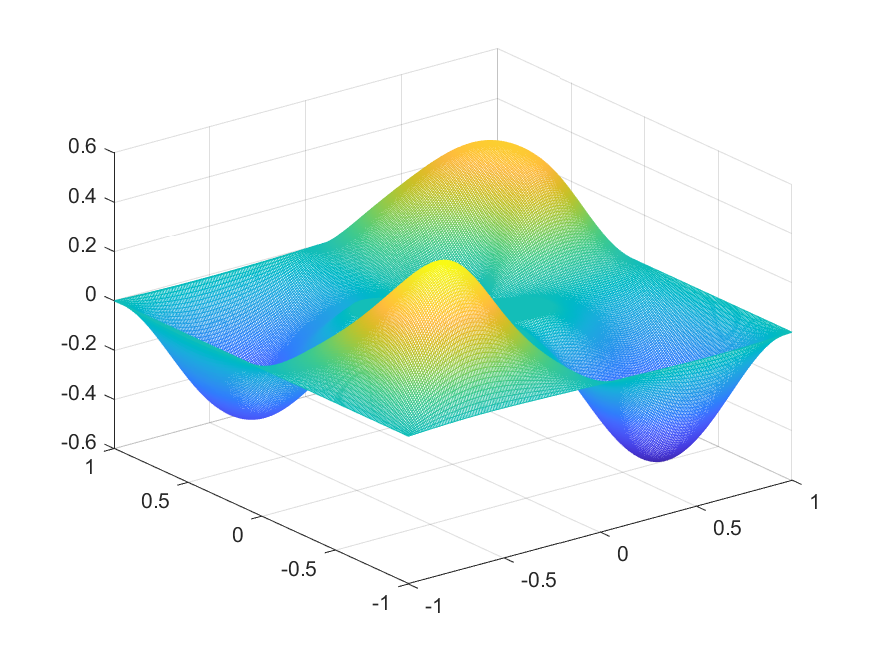}
	\end{minipage}
	\vspace{-10pt}
	\caption{The true solution $u$ (left) and the numerical solution $u_h$ (right) in \textbf{Example 2} with $\beta_1=1,\,\beta_2=10000$. 
	} 
\end{figure}

\begin{table}[H]
	\caption{Numerical results for \textbf{Example 2} with $\beta_1 =10000,\,\beta_2=1$. } 
	\centering
	\begin{tabular}{c c c c c }
		\toprule
		\text{$\frac{1}{h}$}&
		\text{$\|u-u_h\|_{L^2(\Omega)}$}&
		\text{order}&
		\text{$|u-u_h|_{H^1(\Omega)}$}&
		\text{order}\\
		\midrule
		16 & 1.2702e-3 &  & 8.6556e-2 &  \\
		32 & 3.4194e-4 & 1.8933 & 4.5593e-2 & 0.9248 \\
		64  & 9.0480e-5 & 1.9181 & 2.3468e-2 & 0.9581 \\
		128 & 2.3155e-5 & 1.9663 & 1.1884e-2 & 0.9817 \\
		256 & 5.9911e-6 & 1.9504 & 5.9752e-3 & 0.9919 \\
		\bottomrule
	\end{tabular}\label{e2:1}
\end{table}
\begin{table}[H]
	\caption{Numerical results for \textbf{Example 2} with $\beta_1 =100,\,\beta_2=1$. } 
	\centering
	\begin{tabular}{c c c c c }
		\toprule
		\text{$\frac{1}{h}$}&
		\text{$\|u-u_h\|_{L^2(\Omega)}$}&
		\text{order}&
		\text{$|u-u_h|_{H^1(\Omega)}$}&
		\text{order}\\
		\midrule
		16 & 1.2628e-3 &  & 8.6542e-2 &  \\
		32 & 3.4036e-4 & 1.8915 & 4.5583e-2 & 0.9249 \\
		64 & 9.0027e-5 & 1.9186 & 2.3470e-2 & 0.9577 \\
		128 & 2.3101e-5 & 1.9624 & 1.1885e-2 & 0.9817 \\
		256 & 5.8597e-6 & 1.9791 & 5.9742e-3 & 0.9924 \\
		\bottomrule
	\end{tabular}\label{e2:2}
\end{table}
\begin{table}[H]
	\caption{Numerical results for \textbf{Example 2} with $\beta_1 =1,\,\beta_2=1$. } 
	\centering
	\begin{tabular}{c c c c c }
		\toprule
		\text{$\frac{1}{h}$}&
		\text{$\|u-u_h\|_{L^2(\Omega)}$}&
		\text{order}&
		\text{$|u-u_h|_{H^1(\Omega)}$}&
		\text{order}\\
		\midrule
		16 & 3.5408e-3 &  & 2.5773e-1 &  \\
		32 & 9.0846e-4 & 1.9626 & 1.3056e-1 & 0.9812 \\
		64 & 2.3046e-4 & 1.9789 & 6.5703e-2 & 0.9906 \\
		128 & 5.7997e-5 & 1.9905 & 3.2932e-2 & 0.9965 \\
		256 & 1.4545e-5 & 1.9955 & 1.6486e-2 & 0.9983 \\
		\bottomrule
	\end{tabular}\label{e2:3}
\end{table}
\begin{table}[H]
	\caption{Numerical results for \textbf{Example 2} with $\beta_1 =1,\,\beta_2=100$. } 
	\centering
	\begin{tabular}{c c c c c }
		\toprule
		\text{$\frac{1}{h}$}&
		\text{$\|u-u_h\|_{L^2(\Omega)}$}&
		\text{order}&
		\text{$|u-u_h|_{H^1(\Omega)}$}&
		\text{order}\\
		\midrule
		16 &3.4256e-3 &  & 2.4245e-1 &  \\
		32 & 8.7868e-4 & 1.9629 & 1.2230e-1 & 0.9872 \\
		64 & 2.2201e-4 & 1.9847 & 6.1368e-2 & 0.9949 \\
		128 & 5.6070e-5 & 1.9853 & 3.0715e-2 & 0.9986 \\
		256 & 1.4028e-5 & 1.9990 & 1.5366e-2 & 0.9991 \\
		\bottomrule
	\end{tabular}\label{e2:4}
\end{table}
\begin{table}[H]
	\caption{Numerical results for \textbf{Example 2} with $\beta_1 =1,\,\beta_2=10000$. } 
	\centering
	\begin{tabular}{c c c c c }
		\toprule
		\text{$\frac{1}{h}$}&
		\text{$\|u-u_h\|_{L^2(\Omega)}$}&
		\text{order}&
		\text{$|u-u_h|_{H^1(\Omega)}$}&
		\text{order}\\
		\midrule
		16 & 3.4329e-3 &  & 2.4247e-1 &  \\
		32 & 8.8108e-4 & 1.9621 & 1.2231e-1 & 0.9872 \\
		64 &  2.2361e-4 & 1.9783 & 6.1376e-2 & 0.9948 \\
		128 &  5.6284e-5 & 1.9902 & 3.0715e-2 & 0.9987 \\
		256 &  1.4282e-5 & 1.9785 & 1.5369e-2 & 0.9989 \\
		\bottomrule
	\end{tabular}\label{e2:5}
\end{table}
\subsection{Example 3: Test on Interface Position Variation}
Examples 1 and 2 focus on fixed interface geometries to assess approximation capability and geometric robustness. Example 3 instead evaluates robustness to positional variations, examining how the method performs as the interface shifts—altering element cuts (e.g., small cuts or misalignments with the mesh)—with a fixed mesh. The computational domain is $\Omega$ is $(-1,1)\times(-1,1).$ The interface is a circle centered at the $(-t,0)$ with radius $r=0.6,$ i.e.,
\begin{equation*}
	\varPhi_{\Gamma}(x,y) = (x+t)^2+y^2-(0.6)^2.
\end{equation*}
The exact solution is chosen as follows
\begin{equation*}
	u = \frac{1}{\beta}\varPhi_{\Gamma}(x,y)(x^2-1)(y^2-1).
\end{equation*}
We fixed the mesh size with $h=1/20,$ and $-2h\le t\le 2h,$  in increments of $h^2.$ Numerical results are shown in Figures \ref{fig:movinginterface1}–\ref{fig:movinginterface2}.
\begin{figure}[H]
	\vspace{-10pt}
	\centering
	\includegraphics[scale=0.6]{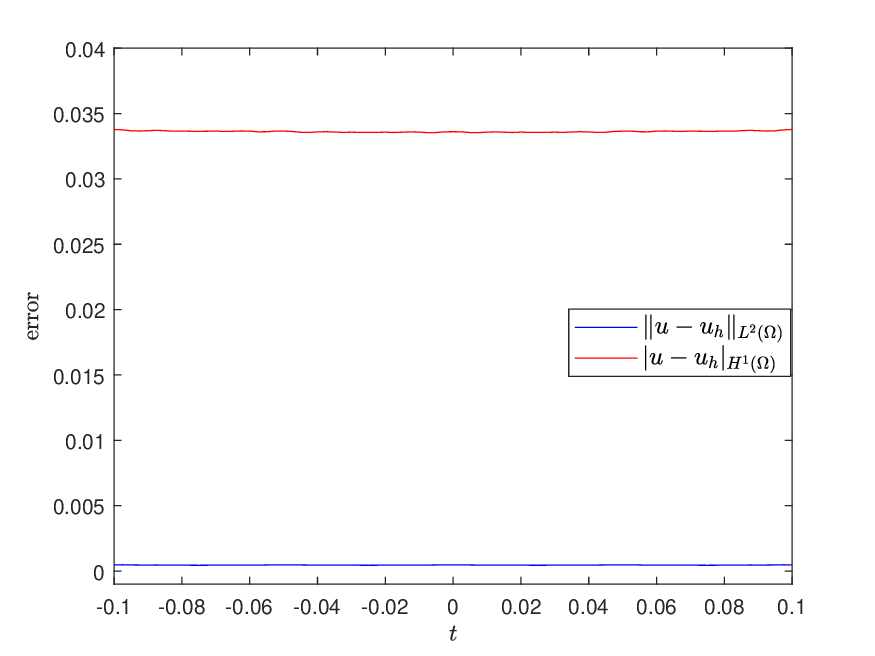} 
	\vspace{-10pt}
	\caption{Finite element errors vary with the position of the interface with $\beta_1=10000,\,\beta_2=1$.}\label{fig:movinginterface1}
\end{figure}
\begin{figure}[H]
	\vspace{-10pt}
	\centering
	\includegraphics[scale=0.6]{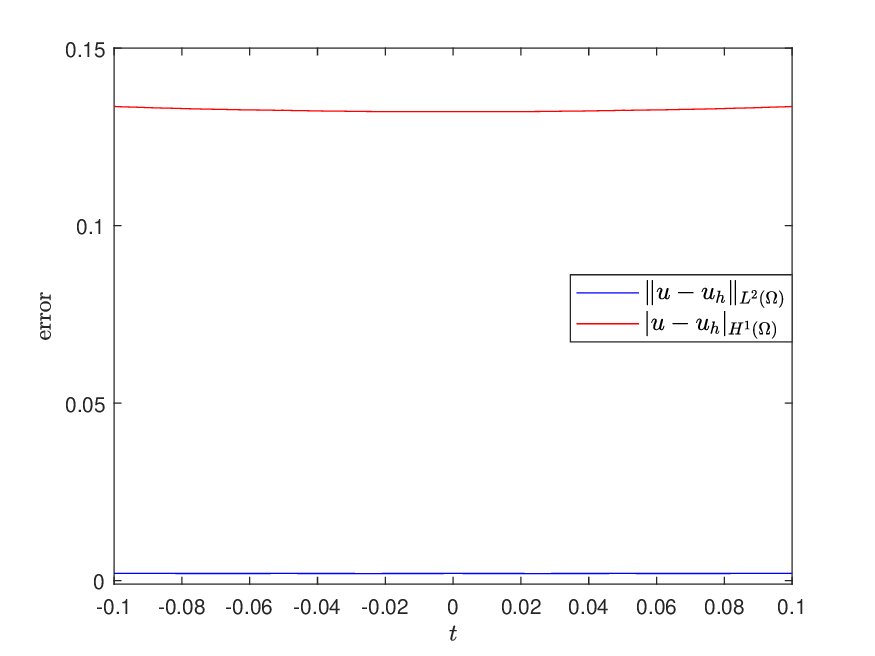} 
	\vspace{-10pt}
	\caption{Finite element errors vary with the position of the interface with $\beta_1=1,\,\beta_2=10000$.}\label{fig:movinginterface2}
\end{figure}
\section{Conclusion}

In this work, we proposed a new nonconforming \(P_1\) finite element method for solving elliptic interface problems. The method is constructed on a locally anisotropic mixed mesh, which was first introduced in our earlier work~\cite{Hu2021optimal}. The present results further demonstrate the effectiveness of this type of mesh in accurately resolving interface geometry while maintaining computational simplicity.

An interesting observation is that the proposed nonconforming element reduces to the standard Crouzeix–Raviart element on triangular elements and to the Park–Sheen element on quadrilateral elements \cite{Park2003P1}. This structure naturally accommodates the use of hybrid meshes and may be beneficial in other applications where elements of different shapes need to be coupled effectively.

We also established interpolation error estimates on quadrilateral elements satisfying the RDP condition, which play an important role in the convergence analysis. More importantly, we developed a novel consistency error estimate for nonconforming elements that removes the quasi-regularity assumption commonly required in existing methods. 

\section*{Declarations}
\textbf{Conflict of interest} The authors declare that they have no conflict of interest.
	\bibliographystyle{plain}
	\bibliography{refer}

\begin{thebibliography}{10}

\bibitem{Acosta2000}
G.~Acosta and R.~G. Dur\'{a}n.
\newblock Error estimates for {$Q_1$} isoparametric elements satisfying a weak
  angle condition.
\newblock {\em SIAM Journal on Numerical Analysis}, 38:1073--1088, 2000.

\bibitem{An2014}
N.~An and H.~Chen.
\newblock A partially penalty immersed interface finite element method for
  anisotropic elliptic interface problems.
\newblock {\em Numerical Methods for Partial Differential Equations},
  30:1984--2028, 2014.

\bibitem{Apel2001cr}
T.~Apel, S.~Nicaise, and J.~Sch\"oberl.
\newblock Crouzeix-{R}aviart type finite elements on anisotropic meshes.
\newblock {\em Numer. Math.}, 89(2):193--223, 2001.

\bibitem{Babuska1976on}
I.~Babu\v~ska and A.~K. Aziz.
\newblock On the angle condition in the finite element method.
\newblock {\em SIAM J. Numer. Anal.}, 13(2):214--226, 1976.

\bibitem{burman2015cutfem}
E.~Burman, S.~Claus, P.~Hansbo, M.~G Larson, and A.~Massing.
\newblock Cutfem: discretizing geometry and partial differential equations.
\newblock {\em International Journal for Numerical Methods in Engineering},
  104(7):472--501, 2015.

\bibitem{Cao2022extended}
P.~Cao, J.~Chen, and F.~Wang.
\newblock An extended mixed finite element method for elliptic interface
  problems.
\newblock {\em Computers And Mathematics with Applications}, 113:148--159,
  2022.

\bibitem{chen2017interface}
L.~Chen, H.~Wei, and M.~Wen.
\newblock An interface-fitted mesh generator and virtual element methods for
  elliptic interface problems.
\newblock {\em Journal of Computational Physics}, 334:327--348, 2017.

\bibitem{Chen1997fully}
Z.~Chen and R.~E. Ewing.
\newblock Fully discrete finite element analysis of multiphase flow in
  groundwater hydrology.
\newblock {\em SIAM J. Numer. Anal.}, 34(6):2228--2253, 1997.

\bibitem{Chen2023an}
Z.~Chen and Y.~Liu.
\newblock An arbitrarily high order unfitted finite element method for elliptic
  interface problems with automatic mesh generation.
\newblock {\em Journal of Computational Physics}, 491:Paper No. 112384, 24,
  2023.

\bibitem{Chen2009the}
Z.~Chen, Y.~Xiao, and L.~Zhang.
\newblock The adaptive immersed interface finite element method for elliptic
  and {M}axwell interface problems.
\newblock {\em Journal of Computational Physics}, 228:5000--5019, 2009.

\bibitem{Chou2012immersed}
S.~H. Chou.
\newblock An immersed linear finite element method with interface flux
  capturing recovery.
\newblock {\em Discrete and Continuous Dynamical Systems-Series B},
  17:2343--2357, 2012.

\bibitem{Fries2010extended}
T.-P. Fries and T.~Belytschko.
\newblock The extended/generalized finite element method: an overview of the
  method and its applications.
\newblock {\em International journal for numerical methods in engineering},
  84(3):253--304, 2010.

\bibitem{Gerstenberger2008an}
A.~Gerstenberger and W.~A. Wall.
\newblock An extended finite element method/{L}agrange multiplier based
  approach for fluid-structure interaction.
\newblock {\em Comput. Methods Appl. Mech. Engrg.}, 197(19-20):1699--1714,
  2008.

\bibitem{Hansbo2002unfitted}
A.~Hansbo and P.~Hansbo.
\newblock An unfitted finite element method, based on {N}itsche's method, for
  elliptic interface problems.
\newblock {\em Computer Methods in Applied Mechanics and Engineering},
  191:5537--5552, 2002.

\bibitem{Hansbo2014cut}
P.~Hansbo, M.~G. Larson, and S.~Zahedi.
\newblock A cut finite element method for a {S}tokes interface problem.
\newblock {\em Applied Numerical Mathematics}, 85:90--114, 2014.

\bibitem{Hu2021optimal}
J.~Hu and H.~Wang.
\newblock An optimal multigrid algorithm for the combining ${P}_1$-${Q}_1$
  finite element approximations of interface problems based on local
  anisotropic fitting meshes.
\newblock {\em Journal of Scientific Computing}, 88, 2021.

\bibitem{Ji2023}
H.~Ji.
\newblock An immersed {C}rouzeix-{R}aviart finite element method in 2{D} and
  3{D} based on discrete level set functions.
\newblock {\em Numerische Mathematik}, 153(2-3):279--325, 2023.

\bibitem{Ji2023analysis}
H.~Ji, F.~Wang, J.~Chen, and Z.~Li.
\newblock Analysis of nonconforming {IFE} methods and a new scheme for elliptic
  interface problems.
\newblock {\em ESAIM Math. Model. Numer. Anal.}, 57:2041--2076, 2023.

\bibitem{Kafafy2005three}
R.~Kafafy, T.~Lin, Y.~Lin, and J.~Wang.
\newblock Three-dimensional immersed finite element methods for electric field
  simulation in composite materials.
\newblock {\em Internat. J. Numer. Methods Engrg.}, 64(7):940--972, 2005.

\bibitem{Li1998}
Z.~Li.
\newblock The immersed interface method using a finite element formulation.
\newblock {\em Applied Numerical Mathematics}, 27:253--267, 1998.

\bibitem{Li2003}
Z.~Li, T.~Lin, and X.~Wu.
\newblock New {C}artesian grid methods for interface problems using the finite
  element formulation.
\newblock {\em Numerische Mathematik}, 96:61--98, 2003.

\bibitem{Lin2015}
T.~Lin, Y.~Lin, and X.~Zhang.
\newblock Partially penalized immersed finite element methods for elliptic
  interface problems.
\newblock {\em SIAM Journal on Numerical Analysis}, 53:1121--1144, 2015.

\bibitem{Mao2005convergence}
S.~Mao and S.~Chen.
\newblock Convergence analysis of the rotated {$Q_1$} element on anisotropic
  rectangular meshes.
\newblock {\em Electron. Trans. Numer. Anal.}, 20:154--163, 2005.

\bibitem{Park2003P1}
C.~Park and D.~Sheen.
\newblock $p_1$-nonconforming quadrilateral finite element methods for
  second-order elliptic problems.
\newblock {\em SIAM Journal on Numerical Analysis}, 41(2):624--640, 2003.

\bibitem{Tezduyar2006space}
T.~E. Tezduyar, S.~Sathe, R.~Keedy, and K.~Stein.
\newblock Space-time finite element techniques for computation of
  fluid-structure interactions.
\newblock {\em Comput. Methods Appl. Mech. Engrg.}, 195(17-18):2002--2027,
  2006.

\bibitem{Wang2019a}
N.~Wang and J.~Chen.
\newblock A nonconforming {N}itsche's extended finite element method for
  {S}tokes interface problems.
\newblock {\em J. Sci. Comput.}, 81(1):342--374, 2019.

\bibitem{Xu2016optimal}
J.~Xu and S.~Zhang.
\newblock Optimal finite element methods for interface problems.
\newblock {\em Domain Decomposition Methods in Science and Engineering XXII},
  pages 77--91, 2016.

\bibitem{Zhang2019strongly}
Q.~Zhang, U.~Banerjee, and I.~Babu\v~ska.
\newblock Strongly stable generalized finite element method ({SSGFEM}) for a
  non-smooth interface problem.
\newblock {\em Comput. Methods Appl. Mech. Engrg.}, 344:538--568, 2019.

\end{thebibliography}
\end{document}